\documentclass{amsart}

\usepackage{amsthm,amsrefs,amsfonts}
\usepackage{mathtools}
\usepackage{tikz-cd}
\usepackage{verbatim}
\usepackage{url}
\usepackage{hyperref}
\usepackage{amsmath}
\usepackage{amssymb}
\usepackage{ulem}
\usepackage[all]{xy}

\tikzcdset{ampersand replacement=\&}

\newtheorem{theorem}{Theorem}[section]
\newtheorem{proposition}[theorem]{Proposition}

\newtheorem{lemma}[theorem]{Lemma}

\theoremstyle{remark}

\theoremstyle{definition}
\newtheorem{definition}{Definition}[section]
\newtheorem{example}{Example}[section]
\newtheorem*{Acknowledgment}{Acknowledgment}
\newtheorem*{question}{Question}

\newtheorem*{T1}{Theorem~\ref{U(X,f)}}

\newcommand{\RN}[1]{%
  \textup{\uppercase\expandafter{\romannumeral#1}}%
}

\newcommand{\N}{\mathcal{N}}
\newcommand{\Z}{\mathbb{Z}}
\newcommand{\Q}{\mathbb{Q}}
\newcommand{\R}{\mathbb{R}}

\newcommand{\Loop}{\mathcal{L}}
\newcommand{\sign}{\text{sign}}

\newcommand{\Mor}{\text{Mor}}
\newcommand{\aut}{\text{Aut}}
\newcommand{\tr}{\operatorname{tr}}
\newcommand{\CWfp}{CW_{\text{fp}}}
\newcommand{\id}{\text{id}}
\newcommand{\fr}{\mathrm{fr}}

\begin{document}
\title[Functoriality of the Klein-Williams Invariant and Universality]{Functoriality of the Klein-Williams Invariant and Universality Theory}
\author{Ba\c{s}ak K\"u\c{c}\"uk}
\address{Universit\"at G\"ottingen, Mathematisches Institut, Bunsenstraße 3-5, 37073 Göttingen\\
}
\email{basak.kucuk@mathematik.uni-goettingen.de}

\begin{abstract}
Both the Klein-Williams invariant $ \ell_G(f) $ from \cite{KW2} and the generalized equivariant Lefschetz invariant $ \lambda_G(f) $ from \cite{weber07} serve as complete obstructions to the fixed point problem in the equivariant setting. The latter is functorial in the sense of Definition \ref{functorial}. The first part of this paper aims to demonstrate that $ \ell_G(f) $ is also functorial. The second part summarizes the ``universality" theory of such functorial invariants, developed in \cites{lueck1999, Weber06}, and explicitly computes the group in which the universal invariant lies, under a certain hypothesis. The final part explores the relationship between $ \ell_G(f) $ and $ \lambda_G(f) $, and presents examples to compare $ \ell_G(f) $, $ \lambda_G(f) $, and the universal invariant.
\end{abstract}

\maketitle

\section{Introduction}

The \textit{Lefschetz number} is a classical invariant in algebraic topology, providing an obstruction theory for the \textit{fixed point problem}, which asks whether an endomorphism on a compact ENR space can be homotoped to a fixed point free map. This leads to the well-known Lefschetz fixed point theorem: if an endomorphism has no fixed points, then the Lefschetz number is equal to zero. However, the Lefschetz number is not a complete invariant, as the converse of the Lefschetz fixed point theorem does not always hold; see \cite{brown.converse.fixpt} for details. A refined invariant, the \textit{Reidemeister trace} \cites{Reidemeister1936AutomorphismenVH, Wecken1941}, provides a complete invariant.

Furthermore, the Reidemeister trace is one of the generalized (functorial) Lefschetz invariants, which satisfy both additivity and homotopy invariance properties. A \textit{functorial Lefschetz invariant} is a pair $ (U, u) $, where $ U $ is a functor from the category of endomorphisms of finite CW-complexes to the category of abelian groups. For any object $ f $, there exists an invariant $ u(f) \in U(f) $ such that $ (U, u) $ satisfies a pushout formula for morphisms in the category of endomorphisms of finite CW-complexes. 

This means that the Reidemeister trace is functorial in the sense that it defines a functor from the category of endomorphisms of finite CW-complexes to the category of abelian groups. This notion was introduced by Lück in \cite{lueck1999}, where he developed a ``universality" theory for functorial Lefschetz invariants. More precisely, a Lefschetz invariant $ (U, u) $ is called \textit{universal} if, for any other functorial Lefschetz invariant $ (A, a) $, there exists a unique natural transformation $ \xi: U \to A $ such that  
\[
\xi(f)(u(f)) = a(f)
\]  
for all objects $ f $ in the category of endomorphisms of finite CW-complexes. L\"uck constructed the universal invariant in an abelian group in terms of Grothendieck groups of endomorphisms of finitely generated free modules. One of the main result of this paper provides an explicit computation of this group for the case of the category of endomorphisms on simply-connected spaces, given by below.

\begin{T1}\textit{
Let $X$ be a simply-connected space. The group $U^\Z(X,f)$, defined as the $K$-group
$$
K_0(\phi\text{-}\mathrm{end}_{\mathrm{ff} \, \Z \Pi(X)}),
$$
in which the universal Lefschetz invariant takes values, is independent of the choice of the space $X$ and the map $f$. Moreover, it is isomorphic to the group $U(\Z)$, which is the free abelian group generated by the set of irreducible characteristic polynomials over $\mathbb{Q}$ of integer matrices. That is,
$$
U(\Z) \cong \Z \left[ \{P \in \Z[x] \mid P \text{ is irreducible over } \Q, P(A) = 0 \text{ for some } A \in M_n(\Z)\} \right].
$$}
\end{T1}

After explicitly computing the abelian group $U^\mathbb{Z}(X,f)$, in which the universal invariant takes values, it is natural to ask the following question. This question is known as the \textit{realization problem}, and we provide an answer in Section \ref{sec.realization}; see Theorem \ref{thm.realization}.

\begin{question}
Does there exist a self-map $ f $ such that the universal functorial equivariant Lefschetz invariant $ u^\mathbb{Z}_G(X,f) $ is equal to the any given element $ [A] $, which lies in $[A]\in U^\Z(X,f)$?
\end{question}

In this paper, we are interested in the equivariant version of the functorial Lefschetz invariants. Weber generalized the construction of functorial Lefschetz invariants to the equivariant setting and defined the \textit{functorial equivariant Lefschetz invariant}; see Definition \ref{functorial}. In the equivariant version, the functorial Lefschetz invariant is also a pair $(U_G, u_G)$, consisting of a family of functors $U_G$ from the category of finite proper $G$-CW complexes for a discrete group $G$ to the category of abelian groups. Moreover, Weber improved the universality theory for functorial equivariant Lefschetz invariants in \cite{Weber06}, and proved that $(U^\Z_G,u^\Z_G)$ is universal, which we briefly explain in Section \ref{sec.universality}.

Another obstruction theory for the fixed point problem in the equivariant setting was introduced by Klein and Williams, as described below.

\begin{theorem}{\cite{KW2}}\label{KWfix}Let $f: M \rightarrow M$ be a $G$-map on a closed, smooth $G$-manifold $M$. Then, there exists an invariant
$$
\ell_G(f) \in \Omega_0^{G, \fr}(\Loop_f M),
$$
which vanishes if $f$ is $G$-equivariantly homotopic to a fixed-point-free map. Here,
$$
\Loop_f M = \{ \lambda \colon [0,1] \to M \mid f(\lambda(0)) = \lambda(1) \}
$$
denotes the space of paths twisted by $f$, and $\Omega_0^{G, \fr}(\Loop_f M)$ is the $G$-equivariant framed bordism group of $\Loop_f M$. Conversely, assume that $\ell_G(f) = 0$. Suppose the following conditions hold:
\begin{itemize}
\item $\dim M^H \geq 3$ for all conjugacy classes of subgroups $(H)$ such that the subgroup $H \subset G$ appears as an isotropy group in $M$, and
\item $\dim M^H \leq \dim M^K - 2$ for all conjugacy classes $(H), (K)$ with proper subgroup inclusions $K \subset H$, where $H$ and $K$ are isotropy subgroups of $M$.
\end{itemize}
Then $f$ is $G$-equivariantly homotopic to a fixed-point-free map.
\end{theorem} 

The definition of the functorial equivariant Lefschetz invariant provides a structural framework, making it natural to ask whether the Klein-Williams invariant is an instance of it. Even though the Klein-Williams invariant was originally constructed for smooth $G$-manifolds with finite group actions, it can also be defined for $G$-CW complexes with finite groups $G$. This is because the decomposition of the Klein-Williams invariant $\ell_G(f)$ under the tom Dieck splitting consists of Reidemeister traces, which can be defined on CW complexes; see Theorem \ref{ident3} for the decomposition and \cite{kucuk2025kleinwilliamsconjectureequivariant} for the proof and further details on the Klein-Williams invariant. We prove that the Klein-Williams invariant is indeed a functorial equivariant Lefschetz invariant, satisfying additivity, $G$-homotopy invariance, and compatibility with the induction structure for finite groups $G$, as shown in Proposition \ref{prop.functoriality}.

After proving that the Klein-Williams invariant is a functorial Lefschetz invariant, the natural question arises: does it correspond to the universal invariant? Note that it is not straightforward to define a unique map from $U_G^\Z(X,f)$, where the universal invariant lies, to the twisted loop space $\Omega_0^{G,\fr}(\Loop_fX)$ such that $u^\Z_G(X,f) \mapsto \ell_G(f)$. In the last section, we present examples where we compute both the Klein-Williams and universal invariants, highlighting the complexity of constructing such a map. Nevertheless, these examples demonstrate that $U_G^\mathbb{Z}(X,f)$ and $\Omega_0^{G,\fr}(\Loop_f X)$ are not isomorphic. This implies that there does not exist a map
$$
\Omega_0^{G,\fr}(\Loop_f X) \longrightarrow U_G^\mathbb{Z}(X,f)
$$
which sends the Klein–Williams invariant $\ell_G(f)$ to the universal invariant $u_G^\mathbb{Z}(X,f)$, and such that the composition with the canonical map
$$
U_G^\mathbb{Z}(X,f) \longrightarrow \Omega_0^{G,\fr}(\Loop_f X)
$$
is the identity on $\ell_G(f)$. Consequently, the Klein–Williams invariant cannot serve as the universal functorial equivariant Lefschetz invariant.

Another important functorial equivariant Lefschetz invariant was developed by Weber \cite{Weber06} via a trace map, which sends the universal functorial Lefschetz invariant $ u^\mathbb{Z}_G(X,f) $ to a new invariant called the \textit{generalized equivariant Lefschetz invariant}, denoted by $ \lambda_G(f) $. The construction of $\lambda_G(f)$ is more algebraic compared to the Klein-Williams invariant. Even though $\lambda_G(f)$ and $\ell_G(f)$ are constructed differently, they contain the same information for the fixed point problem. Under the \textit{gap hypothesis}, which is equivalent to the dimension hypothesis in Theorem \ref{KWfix}, both invariants satisfy the equivariant version of the converse of the Lefschetz fixed point theorem. More precisely, Weber proved the following theorem.

\begin{theorem}\cite[6.2]{weber07}\label{thm.weber}
Let $G$ be a discrete group. Let X be a cocompact proper smooth $G$-manifold satisfying the standard gap hypotheses. Let $f \colon X \to X$ be a $G$-equivariant endomorphism. Then the following holds:
If $\lambda_G(f) = 0$, then $f$ is $G$-homotopic to a fixed point free G-map.
\end{theorem}

As a result, it is natural to ask whether $\lambda_G(f)$ and $\ell_G(f)$ are equivalent. In the last section, our basic examples show that they do not yield the same invariant. However, Theorem \ref{thm.KW.lambda} in Section \ref{sec.lambda.invariant} demonstrates that they vanish simultaneously under the given conditions for smooth $G$-manifolds.

The structure of the paper is as follows. Section \ref{sec.functor} shows that the Klein-Williams invariant is indeed a functorial equivariant Lefschetz invariant. In Section \ref{sec.universality}, we provide a brief explanation of the universality theory of functorial equivariant Lefschetz invariants, followed by a proof of Theorem \ref{U(X,f)}. Furthermore, we address the realization problem for non-equivariant simply-connected spaces in Section \ref{sec.realization}. Section \ref{sec.lambda.invariant} explores the relationship between the Klein-Williams invariant and the generalized equivariant Lefschetz invariant, defined by Weber \cites{Weber06, weber07}, which is constructed as the image of a certain trace map from the universal invariant. We conclude the paper with examples in Section \ref{sec.examples}, where we explicitly compute the Klein-Williams, universal, and generalized equivariant Lefschetz invariants, allowing us to compare these invariants in three different situations.

\begin{Acknowledgment}
This work forms part of the author's PhD research conducted under the supervision of Thomas Schick. The author gratefully acknowledges Thomas Schick for his exceptional guidance, ongoing support, and numerous insightful discussions, particularly his assistance with the proof of Lemma \ref{conj.ideas}. This work was supported by the German Academic Exchange Service (DAAD).
\end{Acknowledgment}

\section{Functoriality of the Klein and Williams Invariant}\label{sec.functor}

Let $ G $ be a finite or discrete group. Denote by $ G\text{-}CW_{\text{fp}} $ the category of finite proper $ G $-CW complexes. A \textit{finite proper} $ G $-CW complex is a $ G $-CW complex in which the group action is proper and the $G$-CW complex has only finitely many cells. Note that a $ G $-CW complex is \textit{proper} if and only if each cell stabilizer is finite, which is equivalent to saying that $ G $ acts properly discontinuously. It is clear that if $ G $ is finite, then any $ G $-CW complex is proper.

Let $ \operatorname{End}(G\text{-}\CWfp) $ denote the category of $ G $-equivariant endomorphisms of finite proper $ G $-CW complexes. That is, the objects are pairs $ (X, f) $, where $ X $ is a finite proper $ G $-CW complex and $ f\colon X \to X $ is an equivariant map. A morphism $ \alpha\colon (X, f) \to (Y, g) $ between two objects is a map $ \alpha\colon X \to Y $ satisfying $ \alpha \circ f = g \circ \alpha $.

This section demonstrates that the Klein-Williams invariant serves as a functorial equivariant Lefschetz invariant on the family of categories finite $G$-CW complexes for finite groups. Weber introduced this concept (Definition 2.3 in \cite{Weber06}) for discrete groups $ G $. Before establishing the functoriality of the Klein-Williams invariant, we first state its definition.

\begin{definition}\label{functorial}
A \textit{functorial equivariant Lefschetz invariant} on the family of categories $G \text{-} \CWfp$, consisting of finite proper $G$-CW complexes for discrete groups $G$, is a pair $(\Theta, \theta)$ with the following components:

\begin{itemize}
\item A family $\Theta = \{ \Theta_G \}$ of functors
$$
\Theta_G \colon \text{End}(G\text{-}\CWfp) \to \mathcal{A}b,
$$
such that for every group inclusion $i\colon H \hookrightarrow G$, there exists a group homomorphism
$$
i_*\colon \Theta_H(X, f) \to \Theta_G(\operatorname{ind}_{i} X, \operatorname{ind}_{i} f)
$$
for each $(X, f) \in \text{End}(H\text{-}\CWfp)$. In addition, the following naturality condition holds for any morphism $\alpha\colon (X, f) \to (Y, g)$:
$$
i_* \Theta_H(\alpha) = \Theta_G(\operatorname{ind}_{i} \alpha) i_*.
$$
\item A family $\theta = \{ \theta_G \}$ of functions
$$
\theta_G \colon (X, f) \mapsto \theta_G(X, f) \in \Theta_G(X, f).
$$
\end{itemize}

Moreover, $(\Theta,\theta)$ satisfies the following conditions:
\begin{enumerate}
\item Additivity:\\
For a $G$-pushout with a $G$-cofibration $i_2$,
\[
\begin{array}{c}
\xymatrix{
  (X_0, f_0) \ar[r]^{i_1} \ar[d]_{i_2} \ar[rd]^{j_0} & (X_1, f_1) \ar[d]^{j_1} \\
  (X_2, f_2) \ar[r]_{j_2} & (X, f)
}
\end{array}
\]
the invariant satisfies
\[
\theta_G(X,f) = \Theta_G(j_1)\theta_G(X_1, f_1) + \Theta_G(j_2)\theta_G(X_2, f_2) - \Theta_G(j_0)\theta_G(X_0, f_0).
\]
\item $G$-Homotopy invariance:\\
If $\alpha_0, \alpha_1 \colon (X, f) \to (Y, g)$ are $G$-maps that are $G$-homotopic in $\text{End}(G\text{-}\CWfp)$, then
\[
\Theta_G(\alpha_0) = \Theta_G(\alpha_1): \Theta_G(X,f) \to \Theta_G(Y,g).
\]
\item Invariance under $G$-homotopy equivalence:\\
If $\alpha \colon (X, f) \to (Y, g)$ is a morphism in $\text{End}(G\text{-}\CWfp)$ such that the map $\alpha \colon X \to Y$ is a $G$-homotopy equivalence, then
\[
\Theta_G(\alpha): \Theta_G(X, f) \cong \Theta_G(Y, g)
\quad \text{and} \quad
\theta_G(X,f) \mapsto \theta_G(Y,g).
\]
\item Normalization:
\[
\theta_G(\varnothing, \text{id}_\varnothing) = 0 \in \Theta_G(\varnothing, \text{id}_\varnothing).
\]
\item Inclusions:\\
For every group inclusion $i \colon H \hookrightarrow G$, we have
\[
i_* \theta_H(X, f) = \theta_G(\text{ind}_i X, \text{ind}_i f) \in \Theta_G(\text{ind}_i X, \text{ind}_i f).
\]
\end{enumerate}
\end{definition}

Before demonstrating that the Klein–Williams invariant $\ell_G(f)$, which is stated in Theorem \ref{KWfix}, defines a functorial equivariant Lefschetz invariant in the sense of Definition \ref{functorial}, we first present the following result, which provides a detailed decomposition of the Klein–Williams invariant (see \cite{kucuk2025kleinwilliamsconjectureequivariant} for the details).

\begin{theorem} \label{ident3}
There exists an isomorphisms between abelian groups that gives an explicit decomposition of the Klein-Williams invariant $\ell_G(f)$.
\begin{align*}
   \Omega_0^{G,\fr} (\Loop_f M) & \to \bigoplus_{(H)} \left( \bigoplus_i \Z [\pi_1(M^H_i,*)_{f^H_i}] \right) \Big/ W_GH\\
    \ell_G(f) & \mapsto \mathop{\oplus}_{(H)} \left( \mathop{\oplus}_i \overline{R(f^H_i)} \right),
\end{align*}
where each $R(f_i^H)$ is the Reidemeister trace of the induced map $f^H_i: M^H_i \to M^H_i$, and $\overline{R(f^H_i)}$ is the quotient of the Weyl group $W_GH$, defined by $N_G(H)/H$.
\end{theorem}

The set $\pi_1(M^H_i,*)_{f^H_i}$ in the theorem above denotes the set of fundamental group elements modulo the \textit{twisted conjugacy} relation, which is given by the equivalence classes:
$$
 \beta \sim  \alpha \beta \phi(\alpha)^{-1},
$$
for all $\alpha, \beta \in \pi_1(M^H_i,*)$, where $\phi$ is the homomorphism induced by the map $f^H_i$.

The \textit{Reidemeister trace} $R(f)$ is defined as a sum over the fixed point classes of $f$, with coefficients given by their fixed point indices. The definition of the fixed point index and related concepts can be found in \cites{jiangbook, brownfix, DOLD19651}.

The \textit{fixed point classes} of a map $f$, denoted by $\operatorname{Fix}^c(f)$, form an equivalence relation on $\operatorname{Fix}(f)$. Two fixed points $x$ and $y$ belong to the same class if and only if there exists a path $\alpha$ from $x$ to $y$ such that $\alpha \simeq f(\alpha)$ relative to the endpoints.

\begin{definition}
Let $f\colon M \to M$ be a continuous map on a compact manifold $M$. The \textit{Reidemeister trace} $R(f)$ is defined as the image of the class
$$
\sum_{[x] \in \operatorname{Fix}^c(f)} i(f,[x]) [x]
$$
under the canonical injection
$$
\mathbb{Z} [\operatorname{Fix}^c(f)]  \hookrightarrow \mathbb{Z} [\pi_1(M,*)_f],
$$
where $i(f,[x])$ denotes the fixed point index of $f$ on an open set $U$ that contains all fixed points in the class $[x]$.
\end{definition}

We are now ready to prove that the Klein–Williams invariant $\ell_G(f)$ is functorial; in particular, it defines a functorial equivariant Lefschetz invariant on the family of categories $G\text{-}\CWfp$ for finite groups $G$.

\begin{proposition}\label{prop.functoriality}
Let $(\Omega, \ell)$ be the pair consisting of:
\begin{itemize}  
    \item A family $\Omega$ of functors  
    \begin{align*}  
        \Omega_0^{G,\fr}\colon \operatorname{End}(G\text{-}CW_\mathrm{fp}) &\to \mathcal{A}b\\  
        (M,f) &\mapsto \Omega_0^{G,\fr}(\Loop_f M)  
    \end{align*}  
    \item A family $\ell$ of functions 
    $$\ell_G\colon (M,f) \mapsto \ell_G(f) \in \Omega_0^{G,\fr}(\Loop_f M).$$
\end{itemize}  
Then, $(\Omega, \ell)$ defines a functorial equivariant Lefschetz invariant on the family of categories $G$-$CW_\text{fp}$ of finite $G$-CW complexes for finite groups $G$.
\end{proposition} 

\begin{proof}
We begin by verifying that the functors $\Omega_0^{G,\mathrm{fr}}$ are compatible with the induction structure. Let $i\colon H \to G$ be a group inclusion. In the category of $G$-CW complexes, the induced $G$-space of an $H$-space $X$ is defined as 
$$
\mathrm{ind}_i X := G \times_H X,
$$
the quotient of $G \times X$ by the equivalence relation $(g i(h), x) \sim (g, hx)$ for $g \in G$, $h \in H$, and $x \in X$. The $G$-action on $G \times_H X$ is given by $g' \cdot [g, x] := [g' g, x]$.

The induced map $\mathrm{ind}_i f$, denoted $\tilde{f}$, is defined by $[g, x] \mapsto [g, f(x)]$. It is straightforward to verify that $\tilde{f}$ is a $G$-equivariant map.

To define the group homomorphism $\Omega_0^{H,\fr}(\Loop_fX) \to \Omega^{G,\fr}_0(\Loop_{\tilde{f}} (G \times_H X))$, we will consider the tom Dieck splitting of $\Omega^{G,\fr}_0(\Loop_{\tilde{f}} (G \times_H X))$ (for the details of the tom Dieck splitting, see \cite[Theorem 1.3, page 246]{may1996equivariant} and \cite[Theorem 7.7, page 154]{Dieck1987}):

\begin{align*}
    \Omega^{G,\fr}_0(\Loop_{\tilde{f}} (G \times_H X))= \bigoplus_{\substack{(N)\\N\leq G}} \Omega_0^{\fr}(EW_{G}N \times_{W_{G}N} (\Loop_{\tilde{f}}(G\times_H X))^N),
\end{align*}
where $W_GN:= N_G(N)/N$ is the Weyl group. When the group $G$ is clear from context, we simply write the Weyl group as $WN$. In this case, the tom Dieck splitting is equal to
\begin{align*}
    \Omega^{G,\fr}_0(\Loop_{\tilde{f}} (G \times_H X))= \bigoplus_{\substack{(N)\\N\leq H}} \Omega_0^{\fr}(EW_{G}N \times_{W_{G}N} \Loop_{\tilde{f}}(G\times_H X^N)).
\end{align*}
This is because $ (\Loop_{\tilde{f}}(G\times_H X))^N = \Loop_{\tilde{f}}(G\times_H X)^N $, and if $ N \leq H $, we have $ (G\times_H X)^N \cong G\times_H X^N $; otherwise, $ (G\times_H X)^N = \emptyset $. The latter holds because if we had an element $ [g,x] $, then we would have $ n[g,x] = [ng,x] = [g,x] $ for all $n\in N$ such that $n \not\in H$. This holds if and only if $ ng = g $, which forces $ n = 1 $, contradicting the non-emptiness of $ (G\times_H X)^N $ when $ N \not\leq H $. The former comes from the fact that $ n[g,x] = [ng,x] = [g,nx]=[g,x] $ if and only if $ x = nx $ for all $ n \in N $.

To define the induced group homomorphism associated with a given inclusion $ i\colon H \to G $, it suffices to construct the map $ \Omega_0^{H,\fr}(\Loop_fX) \to \Omega^{G,\fr}_0(\Loop_{\tilde{f}} (G \times_H X)) $ component-wise, following tom Dieck's splitting. In other words, we need to define the following map:
\begin{align*}
    (i_*)_{(N)}: \Omega_0^{\fr}(EW_{H}N \times_{W_{H}N} (\Loop_{{f}} X^N)) &\to \Omega_0^{\fr}(EW_{G}N \times_{W_{G}N} \Loop_{\tilde{f}}(G\times_H X^N)),
\end{align*}
for all $N \leq H$. We can define it from the map $\Omega_0^{\fr}(\Loop_{{f}} X^N) \to \Omega_0^{\fr}(\Loop_{\tilde{f}}(G\times_H X^N))$ because we have the following isomorphisms.
\begin{align*}
    \Omega_0^{\fr}(EW_{H}N \times_{W_{H}N} (\Loop_{{f}} X^N)) & \cong \Omega_0^{\fr}(\Loop_{{f}} X^N)/W_{H}N\\
    \Omega_0^{\fr}(EW_{G}N \times_{W_{G}N} \Loop_{\tilde{f}}(G\times_H X^N)) & \cong \Omega_0^{\fr}(\Loop_{\tilde{f}}(G\times_H X^N))/W_{G}N.
\end{align*}
The latter can be considered as follows since $W_HN\subseteq W_GN$:
\[
\left(\Omega_0^{\fr}(\Loop_{\tilde{f}}(G\times_H X^N))/W_{H}N\right)\big/W_{G}N
\]
We now define the map
\begin{align*}
    \Omega_0^{\fr}(\Loop_{{f}} X^N) & \to \Omega_0^{\fr}(\Loop_{\tilde{f}}(G\times_H X^N))\\
    [\lambda] & \mapsto \sum_{gH \in G/H} [(g,\lambda)],
\end{align*}
where, the path $(g,\lambda)$ on $G\times_H X^N$ is the extension of the path $\lambda$, which is defined by $t \mapsto [g,\lambda(t)]$ for all $gH \in G/H$. Therefore, we obtained the desired homomorphism
\[
i_*\colon \Omega_0^{H,\fr}(\Loop_fX) \to \Omega^{G,\fr}_0(\Loop_{\tilde{f}} (G \times_H X))
\]

Furthermore, let $\alpha\colon (X,f) \to (Y,g)$ be a morphism in the category of $H$-$CW_\text{fp}$. That is, we have a commutative diagram:
\begin{equation*}
\begin{tikzcd}
    \& X \arrow[r,"f"] \arrow[d,"\alpha"'] \& X \arrow[d,"\alpha"]\\
    \& Y \arrow[r,"g"] \& Y
\end{tikzcd}
\end{equation*}
Then, we have an induced map $\Omega_0^{H,\fr}(\alpha): \Omega_0^{H,\fr}(\Loop_f X) \to \Omega_0^{H,\fr}(\Loop_g Y)$, which we denote as $\alpha_*$ for simplicity. It is defined component-wise with respecting to the tom Dieck splitting:
\begin{align*}
    (\alpha_*)_{(N)}: \Omega_0^{\fr}(EW_{H}N \times_{W_{H}N} \Loop_f X^N) & \to \Omega_0^{\fr}(EW_{H}N \times_{W_{H}N} \Loop_g Y^N)\\
    [\lambda] & \mapsto [\alpha(\lambda)]
\end{align*}
We now describe the naturality condition of the induced maps. Specifically, we aim to show that the following diagram commutes; that is,
$i_*\alpha_*=(\text{ind}_i\alpha)_* i_*$.
\begin{equation} \label{induction.map}
\begin{tikzcd}
    \&  \Omega_0^{H,\fr}(\Loop_f X) \arrow[r,"i_*"] \arrow[d,"\alpha_*"'] \&  \Omega_0^{G,\fr}(\Loop_{\tilde{f}} (G \times_H X)) \arrow[d,"(\text{ind}_i\alpha)_*"]\\
    \&  \Omega_0^{H,\fr}(\Loop_g Y) \arrow[r,"i_*"] \&  \Omega_0^{G,\fr}(\Loop_{\tilde{g}} (G \times_H Y))
\end{tikzcd}
\end{equation}

First, observe that any path $\widetilde{\lambda} \in \Loop_{\tilde{f}}(G \times_H X^N)$ is of the form $\widetilde{\lambda}(t) = (g, \lambda(t))$ for some $gH \in G/H$ and a path $\lambda \in \Loop_f X^N$. This follows from the fact that when $G$ is finite, there is a natural isomorphism $G \times_H X \cong G/H \times X$ for any $H$-space $X$. This identification allows us to define the induced map $(\mathrm{ind}_i \alpha)_*$ explicitly, component-wise, as follows:
\begin{align*}
     \Omega_0^{\fr}(\Loop_{\tilde{f}} (G \times_H X^N))/W_G N & \to  \Omega_0^{\fr}(\Loop_{\tilde{g}} (G \times_H Y^N))/W_G N\\
     [(g,\lambda)] & \mapsto [(g,\alpha(\lambda))].
\end{align*}

This map is well-defined because two elements $[(g,\lambda)]$ and $[(g',\mu)]$ are equal in $\Omega_0^{\fr}(\Loop_{\tilde{f}} (G \times_H X^N))/W_G N$ if and only if $[\lambda]=[\mu]$ in $\Omega_0^{\fr}(\Loop_{f} (X^N))/W_H N$. As a result, the diagram \eqref{induction.map} commutes because
\[(\text{ind}_i\alpha)_*(i_*([\lambda]))= (\text{ind}_i\alpha)_* \left( \sum_{gH\in G/H}[(g,\lambda)]\right)=\sum_{gH\in G/H}[(g,\alpha(\lambda))]=i_*(\alpha_*([\lambda])).\]
This proves the compatibility with the induction structure.

Clearly, we have a family $\ell$ of functions $\ell_G$, which assigns to given object $(X,f)$ in $\text{End}(G\text{-}\CWfp)$ an element $\ell_G(f) \in \Omega_0^{G,\fr}(\Loop_fX)$. Now, we will verify the properties of the functorial equivariant Lefschetz invariant.

1. Additivity: For a $G$-pushout given as in the Definition \ref{functorial}, we need to show that $\ell_G(f)=(j_1)_*\ell_G(f_1)+(j_2)_*\ell_G(f_2)-(j_0)_*\ell_G(f_0)$.

From \cite[Theorem 4.1]{Ferrario1999}, we know that the generalized Lefschetz number (also known as Reidemeister trace) has the additivity property. That is,
\begin{align*}
    R(f^H)=(j_1^H)_*R(f_1^H)+(j_2^H)_*R(f_2^H)-(j_0^H)_*R(f_0^H).
\end{align*}

It is obvious that if the Reidemeister trace $R(f^H)$ is mapped to $\overline{R(f^H)}$ under the map $\Z[\pi_0(\Loop_f X^H)] \to \Z[\pi_0(\Loop_f X^H)]/WH$, the equality is preserved as follows:
\begin{align*}
    \overline{R(f^H)}=(j_1^H)_*\overline{R(f_1^H)}+(j_2^H)_*\overline{R(f_2^H)}-(j_0^H)_*\overline{R(f_0^H)}
\end{align*}

This implies the desired equality for the additivity condition holds by the decomposition of the Klein-Williams invariant under the tom Dieck splitting, which is given by Theorem \ref{ident3}.

2. $G$-Homotopy invariance: Let $\alpha_0, \alpha_1\colon (X, f) \to (Y, g)$ be two $G$-maps that are $G$-homotopic. Then, we aim to show that
\begin{align*}
    (\alpha_0)_* = (\alpha_1)_* \colon \Omega_0^{G,\fr}(\Loop_f X) \to \Omega_0^{G,\fr}(\Loop_g Y).
\end{align*}

Let $\alpha_t \colon X \to Y$ denote a $G$-homotopy between $\alpha_0$ and $\alpha_1$. Since $\alpha_t$ is $G$-equivariant, its fixed-point restriction $\alpha_t^H \colon X^H \to Y^H$ defines a non-equivariant homotopy for every subgroup $H \leq G$. Now, consider the induced maps $(\alpha_0)_*$ and $(\alpha_1)_*$ component-wise. They are given as follows:
\begin{align*}
    (\alpha_i^H)_*\colon \Z \pi_0(\Loop_f X^H)/WH & \to \Z \pi_0(\Loop_g Y^H)/WH \\
    [\lambda] & \mapsto [\alpha_i(\lambda)],
\end{align*}
where $i=0,1$ and $\lambda$ is any path on $X$ from $x$ to $f(x)$. Now, we will show that $\alpha_0(\lambda)$ and $\alpha_1(\lambda)$ are in the same path component; thus they correspond the same element in the target of the map. If there exits a path $\gamma$ from $\alpha_0(x)$ to $\alpha_1(x)$ such that $\gamma * \alpha_1(\lambda) \simeq \alpha_0(\lambda) * g(\gamma)$, then we are done. We can choose $\gamma$ as $\alpha_t(x)$ and define the homotopy $ H: I\times I \to Y^H$ as follows:
\begin{align*}
    H(t,s)=\alpha_{ts}(\lambda(t(1-s))) * \alpha_{1+(1-t)(s-1)}(\lambda(1-s(1-t)))
\end{align*}
One can check that this is well-defined, continuous map, so it gives a homotopy from $s=0$ to $s=1$:
\begin{itemize}
    \item $s=0$: $H(t,0)= \alpha_0(\lambda(t)) * \alpha_t(\lambda(1))= \alpha_0(\lambda) * \alpha_t(f(x))$,
    \item $s=1$: $H(t,1)=\alpha_t(\lambda(0)) * \alpha_1(\lambda(t))= \alpha_t(x) * \alpha_1(\lambda).$
\end{itemize}
Notice that $\alpha_t(f(x)) = g(\alpha_t(x))$ since $\alpha_t\colon (X^H,f) \to (Y^H,g)$ is a morphism. Therefore, we obtain that 
\[
\alpha_t(x) * \alpha_1(\lambda) \simeq \alpha_0(\lambda) * g(\alpha_t(x)).
\]
This implies that $ (\alpha_0^H)_*= (\alpha_1^H)_*$ for all $(H)$; and hence, $(\alpha_0)_*= (\alpha_1)_*$.

3. Invariance under $G$-homotopy equivalence: Let $\alpha\colon (X,f) \to (Y,g)$ be a morphism in $\text{End}(G\text{-}\CWfp)$ such that $\alpha\colon X \to Y$ is a $G$-homotopy equivalence, then we have a $G$-homotopy inverse $\beta\colon (Y,g) \to (X,f)$ of $\alpha$ such that $\alpha \circ \beta \simeq \id_Y$ and $\beta \circ \alpha \simeq \id_X$. Thus, from previous property, we can directly conclude the following induced map is an isomorphism:
\begin{align*}
    \alpha_*\colon \Omega_0^{G,\fr}(\Loop_f X) \xrightarrow{\cong} \Omega_0^{G,\fr}(\Loop_g Y)
\end{align*}

Therefore, we only need to show that $\alpha_*\colon \ell_G(f) \mapsto \ell_G(g)$. Again, we will consider the map $\alpha_*$ component-wise. By Theorem \ref{ident3}, the projection of $\ell_G(f)$ under tom Dieck splitting is $\overline{R(f^H)}$ for each conjugacy class $(H)$ of subgroups $H\leq G$. As a result, it is enough to show that $\alpha_*(R(f^H))=R(g^H)$ under the map $\Omega_0^{\fr}(\Loop_f X^H) \to \Omega_0^{\fr}(\Loop_g Y^H)$.

Since $\alpha$ is a $G$-homotopy equivalence, $\alpha^H\colon (X^H,f^H) \to (Y^H,g^H)$ is a homotopy equivalence with homotopy inverse $\beta^H\colon  (Y^H,g^H) \to (X^H,f^H)$. Also, note that we have
$\alpha^H \circ f^H = g^H \circ \alpha^H$, and this implies that $f^H \simeq \beta^H \circ g^H \circ \alpha^H$. Therefore, 
$$
R(f^H)=R(\beta^H \circ g^H \circ \alpha^H),
$$
by the homotopy invariance property of Reidemeister traces (see \cites{husseini82, geoghegan-handgeotop} for the details). Reidemeister traces also satisfy commutative property, i.e., 
$$
\alpha_*(R(\beta^H \circ g^H \circ \alpha^H))=R(\beta^H \circ \alpha^H \circ g^H)
$$
by \cite[Proposition 1.12]{husseini82}. Thus, we obtain that
\[
\alpha_*(R(f^H))=\alpha_*(R(\beta^H \circ g^H \circ \alpha^H))=R(\beta^H \circ \alpha^H \circ g^H)=R(g^H).
\]

4. Normalization: It is trivial.

5. Inclusions: Let $i\colon H \to G$ be an inclusion of groups, and let $(X, f)$ be an object in $\mathrm{End}(H\text{-}\CWfp)$. Our goal is to show that $i_*(\ell_H(f)) = \ell_G(\tilde{f})$, where $\tilde{f} = \mathrm{ind}_i f$. By Theorem \ref{ident3}, it suffices to verify that
\begin{align*}
    (i_*)_{(N)}\colon \Z \pi_0( \Loop_f X^N)/W_HN & \to \Z \pi_0( \Loop_{\tilde{f}} (G \times_H X^N))/W_G N\\
    \overline{R(f^N)} & \mapsto \overline{R(\tilde{f}^N)},
\end{align*}
for each subgroup $N \subseteq H$. Suppose that $X^N = \bigcup_{i=1}^n X^N_i$, where each $X^N_i$ is a connected component of $X^N$. It suffices to show that
$$
R(\tilde{f}_i) = i_*(R(f_i)),
$$
where $f_i = f|_{X^N_i}$ and $\tilde{f}_i = \tilde{f}|_{G \times_H X^N_i}$, as this implies that 
$$
R(\tilde{f}^N) = i_*(R(f^N)) =\sum_{[x]\in \text{Fix}^c(f^N)} \sum_{gH \in G/H} i(f^N,[x])[(g,x)],
$$ 
and therefore; one can conclude that $ \overline{R(\tilde{f}^N)}= (i_*)_{(N)}(\overline{R(f^N)}).$

Since $G$ is a finite group, we have $G\times_H X^N \cong G/H \times X^N$. Therefore, $G\times_H X^N$ has $|G/H|n$-many connected components:
\[
G/H\times X^N = \bigcup_{i=1}^n \bigcup_{gH\in G/H} \{gH\} \times X^N_i.
\]
Therefore, we have 
\[
R(\tilde{f}_i) = \sum_{gH\in G/H} R({\tilde{f}_i}|_{\{gH\}\times X^N_i}).
\]
Clearly, the number of fixed-point classes of $ f_i $ and $ \tilde{f}_i|_{\{gH\}\times X^N_i} $ are equal for each $ gH \in G/H $. More precisely, a class $ [x] $ belongs to the set of fixed-point classes of $ f_i $ if and only if $ [g,x] $ belongs to the set of fixed-point classes of $ \tilde{f}_i|_{\{gH\}\times X^N_i} $. Therefore, if 
\[
R(f_i)=\sum\limits_{[x]\in \text{Fix}^c(f_i)}i(f_i,[x])[x],
\]
\[
R(\tilde{f}_i|_{\{gH\}\times X^N_i})=\sum_{[x]\in \text{Fix}^c(f_i)}i(\tilde{f}_i|_{\{gH\}\times X^N_i},[(g,x)])[(g,x)].
\]
Also, by multiplicativity rule of index (see \cite[Chapter \RN{1}, 3.5 ]{jiangbook}), we have
\[
i(\tilde{f}_i|_{\{gH\}\times X^N_i},[(g,x)])=i(f_i,[x]).
\]
This is because $\tilde{f}_i|_{\{gH\}\times X^N_i}=c_g \times f_i$, where $c_g$ is constant map on $\{gH\}$. Thus,
\[
R(\tilde{f}_i|_{\{gH\}\times X^N_i})=\sum_{[x]\in \text{Fix}^c(f_i)}i(f_i,[x])[(g,x)].
\]
As a result, we obtain that
\[
R(\tilde{f}_i)= \sum_{gH\in G/H} \sum_{[x]\in \text{Fix}^c(f_i)}i(f_i,[x])[(g,x)]
\]
This is equal to $i_*(R(f_i))$ because 
\[
i_*\colon \sum_{[x]\in \text{Fix}^c(f_i)}i(f_i,[x])[x] \mapsto \sum_{gH\in G/H}\sum_{[x]\in \text{Fix}^c(f_i)}i(f_i,[x])[(g,x)].
\]
This finishes the proof.
\end{proof}

Now, we consider the following result in \cite[Theorem 4.7]{kucuk2025kleinwilliamsconjectureequivariant}: Given an equivariant self-map $ f $ on a $ G $-space $ X $,  
$$
R(f) = 0 \quad \text{if and only if} \quad \overline{R(f)} = 0,
$$
where $ \overline{R(f)} $ is the image of $ R(f) $ under the quotient map:  
\[
\Z[\pi_1(X, *)_f] \to \Z[\pi_1(X, *)_f]/G.
\]  
It is clear that $ R(f) = 0 $ implies $ \overline{R(f)} = 0 $. For the converse, the following proposition was used in \cite{kucuk2025kleinwilliamsconjectureequivariant}.  

\begin{proposition}\label{indicies.orbit}
Given an equivariant map $ f\colon X \to X $, where $ X $ is a $ G $-space, the fixed-point indices at points in the same orbit are equal. That is, for all $ g \in G $,  
\[
i(f, x) = i(f, gx)
\]  
\end{proposition}

We now give an alternative proof of the Proposition \ref{indicies.orbit} by using the functoriality of the Klein-Williams invariant. By the $G$-homotopy equivalence property of the definition of the functorial equivariant Lefschetz invariant, we obtained that
\[
\alpha_*(\ell_G(f))=\ell_G(g),
\]
where $\alpha\colon (X,f)\to(Y,g)$ is an automorphism in the category of $\text{End}(G\text{-}\CWfp)$. That is, it is a morphism in $\text{End}(G\text{-}\CWfp)$, and $\alpha\colon X\to Y$ is a homotopy equivalence.

Note that the Reidemeister trace also satisfies this property. In fact, this is a special case of the functoriality of the Klein-Williams invariant (see the third part of the proof of Theorem \ref{functorial}). Thus, $ \alpha_*(R(f)) = R(g) $. Given that
$$
R(f) = \sum_{[x] \in \text{Fix}^c(f)} i(f, [x])[x] \quad \text{and} \quad R(g) = \sum_{[y] \in \text{Fix}^c(g)} i(g, [y])[y],
$$
we conclude that
$$
\alpha_*(R(f)) = \sum_{[x] \in \text{Fix}^c(f)} i(g, [\alpha(x)])[\alpha(x)] = \sum_{[y] \in \text{Fix}^c(g)} i(g, [y])[y].
$$

\begin{lemma}
There exits a one-to-one correspondence between the fixed point point classes $f$ and $g$ when there exits automorphism $\alpha\colon (X,f) \to (Y,g)$ in $\operatorname{End}(G\text{-}CW_{\mathrm{fp}})$. In other words, the map
\begin{align*}
\operatorname{Fix}^c(f) &\xrightarrow{\alpha} \operatorname{Fix}^c(g)\\
[x] & \mapsto [\alpha(x)]
\end{align*}
is bijection.
\end{lemma}

\begin{proof}
Since $ \alpha\colon (X,f) \to (Y,g) $ is an automorphism, the map $ \alpha\colon  X \to Y $ is a homotopy equivalence. Therefore, there exists a map $ \beta\colon Y \to X $, which satisfies $ \beta \circ \alpha \simeq \operatorname{id}_X $ and $ \alpha \circ \beta \simeq \operatorname{id}_Y $. Our goal is to show that there exists an inverse of the map $ \text{Fix}^c(f) \xrightarrow{\alpha} \text{Fix}^c(g) $, proving its bijectivity. We claim that the following map defined by $\beta$ provides the inverse. 
\begin{align*}
\beta: \text{Fix}^c(g) & \to \text{Fix}^c(f)\\
[y] &\mapsto [\beta(y)]
\end{align*}
Moreover, for all $ [x] \in \text{Fix}^c(f) $, we have 
\[
(\alpha \circ \beta)([x]) = [x].
\]
Let $ H\colon \text{I} \times X \to X $ be a homotopy from $ \operatorname{id}_X $ to $ \beta \circ \alpha $. Thus, there exists a path defined by
\begin{align*}
H(-,x)\colon \text{I} & \to X\\
t & \mapsto H(t,x). 
\end{align*}
This path connects the point $x$ to $\beta(\alpha(x))$. In fact, the path $H(-,x)$ is homotopic to the constant path $c_x$ at $x$. We define an explicit homotopy between them as follows.
\begin{align*}
    \tilde{H}\colon \text{I} \times \text{I} & \to X\\
    (t,s) & \mapsto \tilde{H}(t,s):= H(st,x)
\end{align*}
It is easy to check that $ \tilde{H}(0,s) = H(0,x) = c_x $, and $ \tilde{H}(1,s) = H(s,x) $ defines the path $ H(-,x) $. We denote this path by $ \lambda := H(-,x) $. As $ x $ is a fixed point of $ f $, we have
\[
f(\lambda) \simeq f(c_x)=c_x \simeq \lambda.
\]
This shows that $[x]$ and $[\beta(\alpha(x))]$ are in the same fixed point class. Thus, we obtain
\begin{align*}
\text{Fix}^c(f) &\xrightarrow{\alpha} \text{Fix}^c(g) \xrightarrow{\beta} \text{Fix}^c(f)\\
[x] & \mapsto [\alpha(x)] \mapsto [\beta(\alpha(x))]=[x]
\end{align*}
This gives the one-to-one correspondence.
\end{proof}

\section{Universality of Functorial Equivariant Lefschetz Invariants}\label{sec.universality}

In this section, we explain the universality property among all functorial equivariant Lefschetz invariants. We then explicitly compute the group in which the universal invariant lies under certain conditions. We begin by defining the universal functorial equivariant Lefschetz invariant and provide a brief overview of its construction.

\begin{definition}
A functorial equivariant Lefschetz invariant $(U,u)$ is called \textit{universal} if for any functorial equivariant Lefschetz invariant $(\Theta,\theta)$ there exits a unique family of natural transformations $\xi_G\colon U_G \to \Theta_G$ such that each group homomorphism satisfying that
\begin{align*}
\xi_G{(X,f)}\colon U_G(X,f) & \to \Theta_G(X,f)\\
u_G(X,f) & \mapsto \theta_G(X,f),
\end{align*}
for all $(X,f) \in \operatorname{End}(G\text{-}CW_{\text{fp}})$. Furthermore, the equality
\[
\xi_H\circ i_* = i_* \circ \xi_G
\]
holds for any inclusion $i\colon H\to G$ of discrete groups, where the induced maps $i_*$ are given as follows respectively.
\begin{align*}
i_*\colon U_H(X,f) \to U_G(\operatorname{ind}_i X, \operatorname{ind}_if)\\
i_*\colon \Theta_H(X,f) \to \Theta_G(\operatorname{ind}_i X, \operatorname{ind}_if).
\end{align*}
\end{definition}

The universal equivariant functorial Lefschetz invariant $u^{\Z}_G(X, f)$, defined by Weber \cite{Weber06}, lies in the abelian group $U^{\Z}_G(X, f)$. This abelian group is defined as follows.
\[ 
U^{\Z}_G (X,f) := K_0 (\phi\text{-end}_{\text{ff} \Z \Pi(G,X)}).
\]
This represents the Grothendieck group completion of $\phi$-endomorphisms of finite free $\Z\Pi(G, X)$-modules, where $\Z\Pi(G,X)$-modules and its $\phi$-endomorphisms are explained below.

To provide a precise construction of the universal invariant, we introduce essential notation and definitions, starting with the fundamental category. We refer to \cite{Lck1989TransformationGA} for a detail explanation. The \textit{fundamental category} $\Pi(G, X)$ is defined as follows:
\begin{itemize}
    \item Objects are the $G$-maps $x\colon G/H \to X$ for some subgroup $H \leq G$. We denote such an object by $X(H)$.
    \item Morphisms are the pairs $(\sigma, [w]) \in \mathrm{Mor}(x(H), y(K))$, consisting of a $G$-map $\sigma\colon G/H \to G/K$, and a homotopy class $[w]$ of a map $w\colon G/H \times I \to X$, relative to $G/H \times \partial I$, satisfying $w_1 = x$ and $w_0 = y \circ \sigma$.
\end{itemize}

We define a \textit{$\Z \Pi(G, X)$-module} as a contravariant functor from the fundamental category to the category of $\Z$-modules:
\[
M\colon \Pi(G,X) \to \Z\text{-Mod}.
\]
Let $\phi\colon \Pi(G, X) \to \Pi(G, X)$ be an endofunctor. A natural transformation 
$$
g\colon M \to M \circ \phi
$$
is called a \textit{$\phi$-endomorphism} of the $\Z \Pi(G, X)$-module $M$.

Now, to define the universal invariant, we consider the following contravariant universal covering functor, defined by Lück \cite[Chapter I, Section 8]{Lck1989TransformationGA}.
\begin{align*}
    \widetilde{X}\colon \Pi(G,X) & \to CW\\
    x(H) & \mapsto \widetilde{X^H(x)}\\
    x(H) \to y(K) & \mapsto \widetilde{X^K(y)} \to \widetilde{X^H(x)}.
\end{align*}
Here, $X^H(x)$ denotes the connected component of the fixed point set $X^H$ containing the element $x(1H) \in X^H$, and $\widetilde{X^H(x)}$ is the universal cover of $X^H(x)$.

We define a cellular $\Z \Pi(G, X)$-chain complex $C^c(\widetilde{X})$ as a contravariant functor given by
\[
C^c \circ \widetilde{X}\colon \Pi(G,X) \xrightarrow{\widetilde{X}} CW \xrightarrow{C^c} \Z\text{-}Ch,
\]
where $\widetilde{X}$ is the universal cover functor. This functor yields a finite free $\Z \Pi(G, X)$-chain complex. 

It is clear that any $G$-equivariant map $f\colon  X \to X$ induces an endofunctor
$$
\phi := \Pi(G, f) \colon \Pi(G, X) \to \Pi(G, X),
$$
defined on objects by
$$
x(H) \mapsto f \circ x(H),
$$
and on morphisms by
$$
(\sigma, [w]) \mapsto (\sigma, [f \circ w]).
$$

Consequently, the map $f\colon X \to X$ induces a $\phi$-endomorphism $C^c(\widetilde{f})$ of the cellular $\Pi(G, X)$-chain complex $C^c \circ \widetilde{X}$. This defines the \textit{universal equivariant functorial Lefschetz invariant} $u^\Z_G(X,f)$, which is constructed as the alternating sum of $C^c(\widetilde{f})$ on the cellular $\Z \Pi(G, X)$-chain complex $C^c(\widetilde{X})$, as defined below.
\[
u^{\Z}_G(X,f) := u(C^c(\widetilde{f})).
\]

The group $U^{\mathbb{Z}}_G(X, f)$ associated to an object $(X, f) \in \operatorname{End}(G\text{-}\CWfp)$, where the universal invariant lies, is defined in terms of $K$-groups. Although it is known that these groups are abelian, we have limited information about these “universal” groups. This naturally raises the question: can we explicitly compute the group $U^{\mathbb{Z}}_G(X, f)$?

In general, computing $K$-groups is a difficult task. In our context, the category $\phi\text{-}\mathrm{end}_{\mathrm{ff}\, \mathbb{Z} \Pi(G, X)}$ denotes the category of $\phi$-twisted endomorphisms of finite free $\mathbb{Z} \Pi(G, X)$-modules. Therefore, we restrict our attention for now to the most fundamental case.

We consider the non-equivariant version of $U^\Z_G(X,f)$ with simply-connected spaces. Unraveling the definition of $U^\Z(X,f)=K_0 (\phi\text{-end}_{\text{ff} \Z \Pi(X)})$, one can obtain that this group is isomorphic to the abelian group $U(\Z)$, generated by elements $[A]$, where $A$ is a square matrix with entries in $\Z$. The group satisfies the following relations:

\begin{itemize}
    \item If $A=\begin{pmatrix} B&C \\ 0&D \end{pmatrix}$ for square matrices $B$ and $D$, then $[A] = [B] + [D]$.
    
    \item If $U$ is an invertible matrix over $\Z$ of the same dimension as $A$, then $[UAU^{-1}] = [A]$.
\end{itemize}

This group $U(\Z)$ was defined by Lück \cite{lueck1999}, who developed the universal theory of functorial Lefschetz invariants for the trivial group, before Weber generalized it to the equivariant case \cite{Weber06}.

Even in this basic case, the conjugacy problem in matrix groups remains challenging due to the integer entries. A classical result addressing this problem, due to Latimer and MacDuffee \cite{Latimer1933ACB} and Taussky \cite{Taussky_1949}, is stated as follows.

\begin{theorem}
There is a one-to-one correspondence between the conjugacy classes of integer matrices $A$ satisfying $f(A) = 0$, and the ideal classes of the ring $\mathbb{Z}[\theta]$, where $f(\lambda)$ is a monic polynomial of degree $n$ with integer coefficients that is irreducible over $\mathbb{Q}$, and $\theta$ is a root of $f(\lambda)$.
\end{theorem}

We will apply the theorem above to compute the group $U(\mathbb{Z})$. Before doing so, we first consider the following result on the conjugacy of matrices over $\mathbb{Z}$, which can be found in \cite[Chapter \RN{3}, Section 15]{newman1972integral}.

\begin{theorem}
Let $A \in M_n(\mathbb{Z})$ be any integer matrix. Then $A$ is conjugate over $\mathbb{Z}$ to a block upper triangular matrix
$$
\begin{pmatrix}
A_{11} & A_{12} & \cdots & A_{1p} \\
0 & A_{22} & \cdots & A_{2p} \\
\vdots & \vdots & \ddots & \vdots \\
0 & 0 & \cdots & A_{pp}
\end{pmatrix}
$$
where each diagonal block $A_{ii}$ has a characteristic polynomial that is irreducible over $\mathbb{Q}$, for all $1 \leq i \leq p$.
\end{theorem}

From the results above, we conclude that if $[A] \in U(\Z)$, then  
\[
[A] = [A_{11}] + \ldots + [A_{pp}],
\]
for some matrices $ A_{ii} $, and each of which has irreducible characteristic polynomial over $ \mathbb{Q} $. Moreover, each $ A_{ii} $ has finitely many conjugacy classes, corresponding to the ideal classes of the ring $ \Z[\theta] $, where $ \theta $ is a root of the characteristic polynomial of the matrix $A_{ii}$.

It is well known that two matrices with irreducible characteristic polynomials over a field lie in the same conjugacy class over that field if and only if they have the same characteristic polynomial. Since each element $[A] \in U(\mathbb{Z})$ can be expressed as a sum of classes of matrices whose characteristic polynomials are irreducible over $\mathbb{Q}$, a key question in determining whether the universal group $U(\mathbb{Z})$ admits additional relations is the following:

\begin{question}
Let $A_1$ and $A_2$ are conjugate over $\Q$ but not over $\Z$, and suppose their characteristic polynomials are irreducible over $\Q$. Are there matrices $B_1, B_2$ and $D$ such that 
\begin{align*}
    \begin{pmatrix}
    A_1 & B_1\\
    0 & D        
    \end{pmatrix} 
\hspace{0.1cm} \text{and }
    \begin{pmatrix}
    A_2 & B_2\\
    0 & D        
    \end{pmatrix}
\end{align*}
are conjugate over $\Z$?
\end{question}

We have determined that the answer to this question is affirmative, by the following lemma. This lemma enables us to explicitly describe the structure of the group $U(\Z)$. 

\begin{lemma} \label{conj.ideas}
Let $A$ and $B$ be two integer matrices over $\Z$, and suppose they are conjugate over $\Q$, but not over $\Z$. That is, there exists a matrix $U$, which is invertible over $\Q$ such that
\[
B=UAU^{-1}.
\]
Then, the classes of the matrices are equal in the universal group $U(\Z)$, i.e.,
\[
[A]=[B].
\]
\end{lemma}

\begin{proof}
Without loss of generality, we can assume that $ U $ is an integer matrix and $ U^{-1} $ is a rational matrix, such that the integer matrices $ A $ and $ B=UAU^{-1} $ are conjugate over $ \mathbb{Q} $ but not over $ \mathbb{Z} $.

Let $ X $ be an $n\times n$ integer matrix. Then, observe that the following block matrices  

\[
\begin{pmatrix}
     A & X\\
     0 & UAU^{-1} - UX
\end{pmatrix}
\quad \text{and} \quad 
\begin{pmatrix}
     A - XU & X\\
     0 & UAU^{-1}
\end{pmatrix}
\]  
are $ \mathbb{Z} $-conjugate by the matrix
\[
\begin{pmatrix}
     \mathrm{I} & 0\\
     U & \mathrm{I}
\end{pmatrix}.
\]  
This is because the inverse matrix is given by 

\[
\begin{pmatrix}
     \mathrm{I} & 0\\
     U & \mathrm{I}
\end{pmatrix}^{-1} =  
\begin{pmatrix}
     \mathrm{I} & 0\\
     -U & \mathrm{I}
\end{pmatrix}.
\]  
Furthermore, in the universal group $ U(\mathbb{Z}) $, these block matrices are given by the following relations:

\[
\begin{bmatrix}
     A & X\\
     0 & UAU^{-1} - UX
\end{bmatrix} = [A] + [UAU^{-1} - UX],
\]  

\[
\begin{bmatrix}
     A - XU & X\\
     0 & UAU^{-1}
\end{bmatrix} = [A - XU] + [UAU^{-1}].
\]  
Since the matrices are $ \mathbb{Z} $-conjugate, they belong to the same class in $ U(\mathbb{Z}) $. Therefore, we have the following equality.
\[
[A] + [UAU^{-1} - UX] = [A - XU] + [UAU^{-1}].
\]  
In addition, note that 
\[
U(A-XU)U^{-1}= UAU^{-1}-UX.
\]
Thus, $ A - XU $ and $ UAU^{-1} - UX $ are both integer matrices, and they are conjugate over $ \mathbb{Q} $. Our goal is to show that they are equivalent in $ U(\mathbb{Z}) $; once this is established, the proof will be complete.

One can choose the integer matrix $ U $ such that its entries have no common divisor, without loss of generality. It is then clear that the image of $ U $ contains a primitive vector, which is not divisible by any integer other than $ 1 $.

Let $ U = (\vec{v_1}, \ldots, \vec{v_n}) $, where $ \{ \vec{v_1}, \ldots, \vec{v_n} \} $ is a $ \mathbb{Z} $-basis. Then, there exists a vector $ v = (x_1, \ldots, x_n) $ such that  
\[
Uv = \left( \sum_{i=1}^n v_i^1x_i, \ldots, \sum_{i=1}^n v_i^n x_i \right)
\]
is primitive, where $ \vec{v_i} = (v_i^1, \ldots, v_i^n) $ for all $ i $. Such a primitive vector can be extended to a $ \mathbb{Z} $-basis of $ \mathbb{Z}^n $ since we know that $ \mathbb{Z}^n / \langle Uv \rangle \cong \mathbb{Z}^{n-1} $.

Note that since $ Uv $ is a basis vector, we can choose the integer matrix $ X $ such that it maps the primitive vector $ Uv $ to $ Av $, i.e.,
\[
XUv = Av.
\]
It follows that $ (A - XU)v = 0 $, and therefore, $ A - XU $ has a kernel. Consequently, $ UAU^{-1} - UX $ also has a kernel. Thus, we can $ \mathbb{Z} $-conjugate both $ A - XU $ and $ UAU^{-1} - UX $ into the following form
\[
\begin{pmatrix}
    0 & *\\
    0 & D
\end{pmatrix}.
\]
Now, we assume that $ A - XU $ is $ \mathbb{Z} $-conjugate to  

\[
\begin{pmatrix}
    0 & *\\
    0 & D_1
\end{pmatrix},
\]
and $ UAU^{-1} - UX $ is $ \mathbb{Z} $-conjugate to  

\[
\begin{pmatrix}
    0 & *\\
    0 & D_2
\end{pmatrix}.
\]
Since 

\[
\begin{pmatrix}
    0 & *\\
    0 & D_1
\end{pmatrix}
\quad \text{and} \quad
\begin{pmatrix}
    0 & *\\
    0 & D_2
\end{pmatrix}
\]
are $ \mathbb{Q} $-conjugate, $ D_1 $ and $ D_2 $ are also $ \mathbb{Q} $-conjugate. By applying the same process to the smaller matrices $ D_1 $ and $ D_2 $ and proceeding by induction, we eventually reduce them to the form  

\[
\begin{pmatrix}
    0 & *\\
    0 & d
\end{pmatrix},
\]
where $ d \in \mathbb{Z} $. Therefore, the matrices $ D_1 $ and $ D_2 $ are equivalent in $ U(\mathbb{Z}) $. Consequently, we have the following equalities.
\[
[A - XU] = [0] + [D_1] = [0] + [D_2] = [UAU^{-1} - UX].
\]
Thus, it follows that  
\[
[A] = [UAU^{-1}],
\]
which completes the proof.
\end{proof}

By applying Lemma \ref{conj.ideas} and using the fact that two matrices with irreducible characteristic polynomials over $ \mathbb{Q} $ are conjugate over $ \mathbb{Q} $ if and only if they have the same characteristic polynomial, we arrive at the following conclusion. 

\begin{theorem}\label{U(X,f)}
Let $X$ be a simply-connected space. The group $U^\Z(X,f)$, defined as the $K$-group
$$
K_0(\phi\text{-}\mathrm{end}_{\mathrm{ff} \, \Z \Pi(X)}),
$$
in which the universal Lefschetz invariant takes values, is independent of the choice of the space $X$ and the map $f$. Moreover, it is isomorphic to the group $U(\Z)$, which is the free abelian group generated by the set of irreducible characteristic polynomials over $\mathbb{Q}$ of integer matrices. That is,
$$
U(\Z) \cong \Z \left[ \{P \in \Z[x] \mid P \text{ is irreducible over } \Q, P(A) = 0 \text{ for some } A \in M_n(\Z)\} \right].
$$
\end{theorem}

\section{Realization Problem}\label{sec.realization}

In this section, we aim to address the \textit{realization problem}, which asks whether, for any $ [A] \in U^\mathbb{Z}_G(X,f) $, there exists a self-map $ f $ such that the universal functorial equivariant Lefschetz invariant $ u^\mathbb{Z}_G(X,f) $ is equal to the element $ [A] $.

As explained in the previous chapter, it is quite challenging to provide an explicit computation of the abelian groups $ U^\mathbb{Z}_G(X,f) $. On the other hand, we gave an answer to the group $ U^\mathbb{Z}(X,f) $ in Theorem \ref{U(X,f)} when $ X $ is simply-connected and non-equivariant. Since this group consists of matrices, we were able to find a solution to the realization problem in this case.

We know that when $X$ is simply-connected, and non-equivariant, then the universal group $U^\Z(X,f)$ is isomorphic to $U(\Z)$, which is independent of the choice of space $X$, as $\pi_1(X,*)$ is trivial. Clearly, when $X$ is contractible, the universal invariant always equal to $u^\Z(X,f)=[1]$ since it is a homotopy invariance. On the other hand, we will show that there exists a simply-connected space $X$ and a self-map $f \colon X \to X$ such that $u^\mathbb{Z}(X,f) = [A]$ for any given element $[A] \in U^\mathbb{Z}(X,f)$. This provides a complete solution to the realization problem in the simply-connected, non-equivariant case.

\begin{theorem}\label{thm.realization}
Let $U(\Z)$ be the abelian group, defined by $K_0 (\phi\operatorname{-end}_{\operatorname{ff}\Z })$. Then, for any $[A] \in U(\Z)$, there exits a simply-connected space $X$, and a self-map on $X$ such that $u^\Z(X,f)=[A]$.
\end{theorem}

\begin{proof}
Let $ X = \bigvee_{i=1}^n S_i^2 \vee \bigvee_{j=1}^m S^3_j $, where each $ S^2_i $ and $ S^3_j $ are 2-dimensional and 3-dimensional spheres, respectively.

We define the map $ f: X \to X $ as follows. The $ i $th sphere $ S^2_i $ wraps around all the $ S^2 $ spheres $ a_1^i, \ldots, a_n^i $-many times, respectively. Similarly, $ S^3_j $ wraps around all the $ S^3 $ spheres $ b_1^j, \ldots, b_m^j $-many times. In terms of homology, the induced map  
\[
f_*: H_2(X) \to H_2(X)
\]  
is given by mapping the $ i $th generator to $ (a_1^i, \ldots, a_n^i) \in H_2(X) \cong \mathbb{Z}^n $.  
Furthermore, the induced map  
\[
f_*: H_3(X) \to H_3(X)  
\]  
is defined by mapping the $ j $th generator to $ (b_1^j, \ldots, b_m^j) \in H_3(X) \cong \mathbb{Z}^m $.

Then, the cellular chain complex of $X$ is given as follows.
\begin{align*}
    C_0(X)&=\langle e_0 \rangle\\
    C_1(X)&=\emptyset\\
    C_2(X)&= \langle e_2^1,\ldots, e^n_2\rangle\\
    C_3(X)&=\langle e_3^1, \ldots, e_3^m\rangle
\end{align*}
The cellular map is then given by the following matrices:  
\[
C_0(f) =  
\begin{bmatrix}  
    1  
\end{bmatrix},  
\]

\[
C_2(f) =  
\begin{bmatrix}  
    a_1^1 & \cdots & a_1^n \\  
    \vdots & \ddots & \vdots \\  
    a_n^1 & \cdots & a_n^n  
\end{bmatrix}=A,  
\]

\[
C_3(f) =  
\begin{bmatrix}  
    b_1^1 & \cdots & b_1^m \\  
    \vdots & \ddots & \vdots \\  
    b_m^1 & \cdots & b_m^m  
\end{bmatrix}=B.  
\]  
Now, it is clear to see that $u^\Z(X,f)=u(C(f))=[1]+[A]-[B]$. One can take the matrix $B$ as $[B]=[1]+[B']$ for some matrix $B'$. Thus we obtain that
\[
u^\Z(X,f)=[A]-[B'],
\]
for any matrix $A,B'\in M(\Z)$. This finishes the proof.
\end{proof}

\section{Generalized Equivariant Lefschetz Invariant}\label{sec.lambda.invariant}

Weber developed a generalized equivariant Lefschetz invariant, denoted by $ \lambda_G(f) $, which takes values in an abelian group $ \Lambda_G(X,f) $ in \cite{Weber06}. This invariant arises as the image of the universal functorial Lefschetz invariant $ u^\mathbb{Z}_G(X,f) $ under a suitable trace map
$$
\tr_{G(X,f)}\colon U^\mathbb{Z}_G(X,f) \to \Lambda_G(X,f).
$$

In this section, our goal is to establish a relationship between two equivariant Lefschetz invariants: $ \lambda_G(f) $ and the Klein-Williams invariant $ \ell_G(f) $. Although these two invariants are defined in fundamentally different ways, under certain hypotheses, they encapsulate the same essential information for the equivariant fixed point problem.

We begin with a brief explanation of Weber's invariant $ \lambda_G(f) $; further details can be found in \cites{Weber06, weber07}. The first definition we present is essential for constructing the group $ \Lambda_G(X,f) $. This group extends the classical free abelian group generated by the twisted conjugacy classes of the fundamental group, where the Reidemeister trace is contained.

Let $x$ be an object in the fundamental category $\Pi(G,X)$. Note that we will consider $x$ as the point $x(1H)$ in the fixed point set $X^H$. Furthermore, we denote by $ X^H(x) $ the connected component of $ X^H $ that contains the point $ x(1H) $.

\begin{definition}\label{defnWeber}\cites{weber07, Weber06}
For an object $x$ in $\Pi(G,X)$ with $X^H(x)=X^H(f(x))$, and a morphism $\nu=(\id,[w])$ from the object $f(x)$ to $x$ in $\Pi(G,X)$. Let
\[
\Z[\pi_1(X^H(x),x)_\phi]
\]
be the free abelian group of the set $\pi_1(X^H(x),x)_\phi$ given by
\[
\pi_1(X^H(x),x)/ \phi(\gamma)\alpha\gamma^{-1} \sim \alpha
\]
where $\alpha\in \pi_1(X^H(x),x)$, $\gamma\in \operatorname{Aut}(x)$, and $\phi(\gamma):=\nu f(\gamma)\nu^{-1}\in \operatorname{Aut}(x)$.
\end{definition}

Here, $\operatorname{Aut}(x)$ denotes the set of morphisms from the object $x$ to $x$. It is clear that $\operatorname{Aut}(x)$ has a group structure, and there exits a group extension lying in the following short exact sequence \cite{Lck1989TransformationGA}.
\[
1 \to \pi_1(X^H(x),x) \to \text{Aut}(x) \to WH_x \to 1
\]
where $WH_x$ is the stabilizer group of the Weyl group action $WH$ that acts on the set of components of $X^H$. In other words, $WH_x$ is a subgroup of $WH$ that fixes $X^H(x)$.

Note that Definition \ref{defnWeber} is well-defined since $ \phi(\gamma)\alpha\gamma^{-1} \in \pi_1(X^H(x),x) $ for all $ \gamma \in \operatorname{Aut}(x) $ and $ \alpha \in \pi_1(X^H(x),x) $. This holds because $ \phi(\gamma) $ does not affect the $ WH_x $-component of the morphism $ \gamma $, and $ \pi_1(X^H(x),x) $ is a normal subgroup of $ \operatorname{Aut}(x) $.

Weber also showed that Definition \ref{defnWeber} is independent of the choice of the base point $ x $ and the path $ w $ from $f(x)$ to $x$ in \cite[Lemma 5.2]{Weber06}.

Now, we introduce the definition of the target group $ \Lambda_G(X,f) $, where $ \lambda_G(f) $ is contained, as follows.  
\[
\Lambda_G(X,f):=\bigoplus_{\substack{\overline{x}\in \text{Is}\Pi(G,X)\\ X^H(f(x))=X^H(x)}} \Z [\pi_1(X^H(x),x)_\phi]
\]
Here $\text{Is}\Pi(G,X)$ denotes the set of isomorphism classes of the fundamental category $\Pi(G,X)$. It is shown in \cite[Equation 3.3]{LueckRosenberg03} that we have 
\begin{align*}
    \text{Is}\Pi(G,X) & \xrightarrow{\simeq} \bigsqcup_{(H)} WH \setminus \pi_0(X^H)\\
    \{x: G/H \to X\} & \mapsto WH\cdot X^H(x)
\end{align*}
where $ WH \cdot X^H(x) $ denotes the orbit of the component $ X^H(x) $ of $ X^H $ that contains the point $ x(1H) $ under the $ WH $-action on $ \pi_0(X^H) $. Before defining the invariant $ \lambda_G(f) $, we first establish some notation, following the conventions used in \cites{Weber06, weber07}. Let 
$$
X^{>H}(x)=\{z\in X^H(x)\mid G_z \neq H\}.
$$
Also, we simply denote $f_{|X^H(x)}$ as $f^H(x)$, $f_{|X^{>H}(x)}$ as $f^{>H}(x)$, and $f_{|(X^H(x),X^{>H}(x))}$ as $f_H(x)$.

Let $ \widetilde{f^H(x)} $ denote the lift of $ f^H(x) $ to the universal covering space $ \widetilde{X^H(x)} $. We also define the subset $ \widetilde{X^{>H}(x)} \subseteq \widetilde{X^H(x)} $ as the preimage of $ X^{>H}(x) $ under the covering map, i.e., 
\[
\widetilde{X^{>H}(x)} = p^{-1}(X^{>H}(x)),
\]
where $p\colon \widetilde{X^H(x)} \to X^H(x)$. Furthermore, let the map 
\[ \widetilde{f^{>H}(x)}: \widetilde{X^{>H}(x)} \to \widetilde{X^{>H}(x)} \] 
denote the lift of $ f^{>H}(x) $.

\begin{definition}\cites{Weber06,weber07}
    The \textit{generalized equivariant Lefschetz invariant} $\lambda_G(f)$, lies in $\Lambda_G(f)$, at the summand indexed by $\overline{x}\in \text{Is}\Pi(G,X)$, is defined as follows.
    \[
    \lambda_G(f)_{\overline{x}} := \sum_{p\geq 0} (-1)^p \tr_{\Z\aut(x)}(C_p^c(\widetilde{f^H(x)},\widetilde{f^{>H}(x)})).
    \]
\end{definition}

This trace map $\tr_{\Z\aut(x)}$ is induced by the following projection.
\begin{align*}
    \Z\aut(x) & \to \Z [\pi_1(X^H(x),x)_\phi]\\
    \sum_{g\in\aut(x)}r_g \cdot g &\mapsto \sum_{g\in \pi_1(X^H(x),x)_\phi} r_g\cdot \overline{g}
\end{align*}

For the full definition of the trace map $ \operatorname{tr}_{\mathbb{Z} \operatorname{Aut}(x)} $; we refer to \cite[Definition 5.4]{Weber06}. Note that the trace map $ \operatorname{tr}_{G(X,f)} $, which maps the universal invariant $ u_G^\mathbb{Z}(X,f) $ to $ \lambda_G(f) $, is defined using $ \operatorname{tr}_{\mathbb{Z} \operatorname{Aut}(x)} $ and the Splitting Theorem \cite[Theorem 4.9]{Weber06} for the universal group $ U^\mathbb{Z}_G(X,f) $. We refer to \cite[Sections 4 and 5]{Weber06} for further details.

Both Klein-Williams invariant $\ell_G(f)$ and the generalized equivariant Lefschetz invariant $\lambda_G(f)$ give an obstruction theory for the equivariant fixed point problem under the gap hypothesis: see \cite[Theorem H]{KW2} for the first one, and \cite[Theorem 6.2]{weber07} for the second. Therefore, it is natural to ask that does two invariant give the same information although they defined by using different techniques. Now we present our result that gives the relationship between $\ell_G(f)$ and $\lambda_G(f)$.

\begin{theorem}\label{thm.KW.lambda}
Given an equivariant smooth self-map on a $G$-manifold $M$, the Klein-Williams invariant $\ell_G(f)$ and the generalized equivariant Lefschetz invariant $\lambda_G(f)$ vanishes at the same time. In other words,
\[
\ell_G(f)=0  \quad \text{if and only if} \quad \lambda_G(f)=0.
\]
\end{theorem}

\begin{proof}
We will use the \textit{character map}
$$
\operatorname{ch}_G(M,f) \colon \Lambda_G(M,f) \to \bigoplus_{\overline{x} \in \operatorname{Is} \Pi(G,M)} \mathbb{Q}[\pi_1(M^H(x),x)_\phi],
$$
which is explicitly defined in \cite[Definition 6.2]{Weber06}. Moreover, it is shown that the map $\operatorname{ch}_G(M,f)$ is injective, and satisfies
$$
\operatorname{ch}_G(M,f)(\lambda_G(f))_{\overline{x}} = L^{\mathbb{Q}\operatorname{Aut}(x)}(\widetilde{f^H(x)}).
$$
In the case where $\widetilde{f^H(x)}$ is a smooth map on a connected, simply-connected, cocompact proper $\operatorname{Aut}(x)$-manifold $\widetilde{M^H(x)}$, \cite[Theorem 6.6]{Weber06} shows that the Lefschetz number can be computed as follows.
\[
L^{\Q\aut(x)}(\widetilde{f^H(x)})= \sum_{\substack{WH_x\cdot z \in \\WH_x\setminus\text{Fix}(f^H(x)) }} |(WH_x)_z|^{-1} \deg((\id_{T_z M^H(x)}-D_zf^H(x))^c)\cdot \overline{a}_z 
\]

\[
\in \Q[\pi_1(M^H(x),x)_\phi].
\]
Since each $ WH_x $ is finite and $ M^H(x) $ is a closed subset of the compact manifold $ M $, any $ M^H(x) $ is a cocompact proper $ WH_x $-manifold. This implies that $ \widetilde{M^H(x)} $ is also a cocompact $ \operatorname{Aut}(x) $-manifold, due to the following homeomorphism.
\[
\aut(x) \setminus \widetilde{M^H(x)} \cong WH_x \setminus M^H(x).
\]
Furthermore, all the isotropy groups of $ \operatorname{Aut}(x) $ are finite since the $ \pi_1(M^H(x),x) $-component of $ \operatorname{Aut}(x) $ acts freely on $ \widetilde{M^H(x)} $, and thus the isotropy groups contain only $ WH_x $, which is finite. Therefore, $ \widetilde{M^H(x)} $ is also a proper $ \operatorname{Aut}(x) $-manifold.

Without loss of generality, we can assume that $ f^H(x) $ has only generic fixed points; that is,  
\[
\det (\text{I} - D_z f^H(x)) \neq 0,
\]
for all $z \in \operatorname{Fix}(f^H(x))$. This is justified by the fact that one can always find a representative in the $ G $-homotopy class of $ f $ that satisfies this assumption, since both the Klein-Williams invariant and the generalized equivariant Lefschetz invariant are invariant under $ G $-homotopy.

Now, we first assume that $ \lambda_G(f) = 0 $. Using the homomorphism $ \text{ch}_G(M,f) $, we obtain  
\[
L^{\mathbb{Q}\aut(x)}(\widetilde{f^H(x)}) = 0  
\]  
for all $ \overline{x} \in \text{Is}\Pi(G,M) $. Next, consider the degree of the map $ (\id_{T_z M^H(x)} - D_z f^H(x))^c $, which is defined by
\[
\id_{T_z M^H(x)} - D_z f^H(x)\colon (T_z M^H(x))^c \to (T_z M^H(x))^c,
\]  
where $ (T_z M^H(x))^c $ denotes the one-point compactification of $ T_z M^H(x) $. Note that the degree of this map is equal to
\[
\sign (\det (\text{I} - D_z f^H(x))).
\]  
This holds because the local degree at $ 0 \in T_z M^H(x) $ is equal to 
\[
\deg((\id_{T_z M^H(x)} - D_z f^H(x))^c),
\]
and this local degree can be computed from the sign of the determinant of the Jacobian. Thus, we have the equality  
\[
\deg((\id_{T_z M^H(x)} - D_z f^H(x))^c) = i(f^H(x),z),
\]  
provided that each $ z \in \operatorname{Fix}(f^H(x)) $ is generic. Then,  

\begin{align*}
\sum_{\substack{WH_x\cdot z \in \\WH_x\setminus\text{Fix}(f^H(x)) }} |(WH_x)_z|^{-1} \deg((\id_{T_z M^H(x)}-D_zf^H(x))^c)\cdot \overline{a}_z\\
= \sum_{\substack{WH_x\cdot z \in \\WH_x\setminus\text{Fix}(f^H(x)) }} |(WH_x)_z|^{-1} i(f^H(x),z)\cdot \overline{a}_z=0.
\end{align*}
Note that $a_z\in\pi_1(M^H(x),x)$ is given by the loop $\zeta_z * f(\zeta_z)^{-1} * w^{-1}$, where $\zeta_z$ a path from $x$ to $z$, and $w$ is a path from $x$ to $f(x)$.

Furthermore, there is a one-to-one correspondence between the classes of $ \overline{a}_z \in \pi_1(M^H(x),x)_\phi $ and the fixed point classes $ [z] \in \text{Fix}^c(f^H(x)) $. Since we are considering only generic fixed points, the fixed point index of the class $ [z] $ is given by the following sum. Moreover, by the arguments above, one can conclude that it is equal to zero.
\[
i(f^H(x),[z]) = \sum_{z \in [z]} i(f^H(x),z) \cdot \overline{a}_z = 0,
\]  
for all fixed point classes $ [z] \in \text{Fix}^c(f^H(x)) $ containing $ z $. Consequently, this implies that
\[
R(f^H(x)) = 0.  
\] 
Thus, we have $ \overline{R(f^H(x))} = 0 $ for all conjugacy classes $ (H) $ of subgroups $ H \leq G $. As a result, we conclude that
\[
\ell_G(f) = 0.
\]  
The converse direction follows by applying the same reasoning: Suppose that $ \ell_G(f) = 0 $. Then, we have 
\[
\overline{R(f^H)} = 0 \quad \text{for all } (H).  
\]  
Thus, 
\[
R(f^H(x)) = \sum_{[z] \in \text{Fix}^c(f^H(x))} i(f^H(x),[z]) \cdot \overline{a}_z = 0.
\]  
Since we are considering only generic fixed points, we again have
\[
i(f^H(x),[z]) = \sum_{z \in [z]} i(f^H(x),z) \cdot \overline{a}_z = 0.
\]  
This further implies that 
\[
\sum_{\substack{WH_x \cdot z \in \\ WH_x \setminus \text{Fix}(f^H(x)) }} |(WH_x)_z|^{-1} \deg((\id_{T_z M^H(x)} - D_z f^H(x))^c) \cdot \overline{a}_z = 0.
\]  
Hence, we obtain  
\[
\text{ch}_G(M,f)(\lambda_G(f))_{\overline{x}} = L^{\mathbb{Q}\aut(x)}(\widetilde{f^H(x)}) = 0.  
\]  
By using the injectivity of the character map, we conclude that
\[
\lambda_G(f)_{\overline{x}} = 0 \quad \text{for all } \overline{x} \in \text{Is}\Pi(G,M).  
\]  
Thus, we have $ \lambda_G(f) = 0 $.  
\end{proof}

\section{Examples}\label{sec.examples}

In this section, we construct several examples to explicitly compare the Klein-Williams invariant $ \ell_G(f) $, the generalized equivariant Lefschetz invariant $ \lambda_G(f) $, and the universal invariant $ u^\mathbb{Z}_G(X,f) $.

The first example shows that the universal invariant is equal to the $ 1 \times 1 $-matrix $ [g] $ over $ \mathbb{Z}[\mathbb{Z}_2] $, where $ g $ is the generator of $ \mathbb{Z}_2 $. The fundamental group is trivial, so the trace of $ [g] $ is zero. In this case, both $ \ell_G(f) $ and $ \lambda_G(f) $ vanish.

In the second example, we have non-zero invariants, but $ \ell_G(f) $ and $ \lambda_G(f) $ are different. Although, in this example, they satisfy the relation $\ell_G(f)_x=|WH_x|\lambda_G(f)_x$ for all $\overline{x} \in \text{Is}\Pi(G,X)$, this relationship does not always hold. In fact, the final example illustrates the complexity of this relationship; in other words, these invariants cannot always be computed from each other.

\begin{example}
Let $ X = S^2 $ and $ G = \mathbb{Z}_2 $, the cyclic group of order 2, with generator $ g $. The group acts on the sphere by reflection across the $ xy $-plane as follows.
\[
g \cdot (x_1, x_2, x_3) = (x_1, x_2, -x_3).
\]  
Let $f$ be an equivariant map which is the same as the generator $g$ of the group $\Z_2$.
\[
f: S^2 \to S^2  
\]  
\[
(x_1, x_2, x_3) \mapsto (x_1, x_2, -x_3).  
\]  
Then, the Klein-Williams invariant $ \ell_G(f) $ is given by 
\[
\ell_G(f) = \overline{R(f)}^{\mathbb{Z}_2} + R(f^{\mathbb{Z}_2}).  
\]  
Note that this invariant lies in the following decomposition.
\[
\Omega_0^{\mathbb{Z}_2,\fr}(\Loop_f S^2) = \mathbb{Z}[\pi_1(S^2,x)_f/\mathbb{Z}_2] \oplus \mathbb{Z}[\pi_1((S^2)^{\mathbb{Z}_2},y)_f] \cong \mathbb{Z} \oplus \mathbb{Z}[\mathbb{Z}].  
\]  
Since $ \overline{R(f)}^{\mathbb{Z}_2} \in \mathbb{Z}[\pi_1(S^2,x)_f/\mathbb{Z}_2] \cong \mathbb{Z}[\{1\}] $, we have
\[
\overline{R(f)}^{\mathbb{Z}_2} = R(f) = L(f)[1],  
\]  
where $ L(f) $ is the Lefschetz number of $ f $, which is given by  
\[
L(f) = 1 + (-1)^2\deg(f) = 1 + (-1) = 0.  
\]  
Note that $ R(f^{\mathbb{Z}_2}) = 0 $ since $ f^{\mathbb{Z}_2}\colon S^1 \to S^1 $ is the identity map, which is homotopic to a fixed-point-free map. Therefore, $ \ell_{\mathbb{Z}_2}(f) $ vanishes.

Next, to compute the universal equivariant invariant $ u_G^\mathbb{Z}(X,f) $, we first need to analyze the fundamental category $ \Pi(\mathbb{Z}_2, S^2) $. Objects of $\Pi(\Z_2,S^2)$ can be categories as follows.

If $x\in S^2 - (S^2)^{\Z_2}$, then
\begin{align*}
    x\colon \Z_2/\{1\} & \to S^2\\
    1 & \mapsto x\\
    g & \mapsto g\cdot x
\end{align*}

If $y\in (S^2)^{\Z_2}$, then
\begin{align*}
    y\colon \Z_2/\Z_2 & \to S^2\\
    1 & \mapsto y
\end{align*}

There are four type of morphisms of $\Pi(\Z_2,S^2)$, which are listed below.
$$
\Mor(x_1,x_2)=\{(\text{id},[\text{path}_{x_1}^{x_2}]),(r,[\text{path}_{x_1}^{x_2}])\},
$$
where $\text{id}\colon \Z_2/\{1\}\to \Z_2/\{1\}$ is the identity map, and $r\colon \Z_2/\{1\}\to \Z_2/\{1\}$, defined as $r(1)=g$, $r(g)=1$. These are $\Z_2$-maps, and $\text{path}_{x_1}^{x_2}$ denotes a path from $x_1$ to $x_2$ in $S^2$. Note that any path from $x_1$ to $x_2$ are homotopic in $S^2$.
$$
\Mor(y_1,y_2)=\{(\text{id},[\text{p}_{y_1}^{y}*s^n*\text{p}_{y}^{y_2}]):n\in \Z\},
$$
where $\text{id}:\Z_2/\Z_2\to \Z_2/\Z_2$. Also, $s$ is the generator of $\pi_1((S^2)^{\Z_2},y)=\pi_1(S^1,y)$, which is the loop at $y$, going once around $S^1$ counterclockwise, and $\text{p}_{y}^{y'}$ denotes path from $y$ to $y'$ in $S^1$.
$$
\Mor(x,y)=\{(p,[\text{path}_{x}^{y}])\},
$$
where $p:\Z_2/\{1\}\to \Z_2/\Z_2$ such that $p(1)=1$, $p(g)=1$ is $\Z_2$-map.

$$
\Mor(y,x)= \emptyset
$$
because there is no $\Z_2$-map from $\Z_2/\Z_2$ to $\Z_2/\{1\}$.

We now construct the $\mathbb{Z} \Pi(\mathbb{Z}_2, S^2)$-chain complex $C^c(\widetilde{S^2})$.
To begin, consider a $\mathbb{Z}_2$-CW structure on the 2-sphere $S^2$. This structure includes: Two $0$-cells of orbit type $\mathbb{Z}_2$, denoted 
$$a, b \colon e_0 \to (S^2)^{\mathbb{Z}_2},$$
Two $1$-cells of orbit type $\mathbb{Z}_2$, denoted 
$$k, l \colon e_1 \to (S^2)^{\mathbb{Z}_2},$$
One $2$-cell of orbit type $\{1\}$, denoted by the following map.
\begin{align*}
\theta \colon \mathbb{Z}_2 \times e_2 & \to S^2\\
(1, u) & \mapsto u\\
(g, u) & \mapsto g \cdot u 
\end{align*}
Note that $S^2(x)$, the connected component of $S^2$ contains the point $x$, is equal to $S^2$ itself. Also, we know that $\aut(x) \cong \Z_2$ by the following short exact sequence.
\begin{align*}
    1 \to \pi_1(S^2(x),x) \to \aut(x) \to W1\cong\Z_2 \to 1
\end{align*}
Therefore, an $\aut(x)$-CW structure on $\widetilde{S^2(x)}$ is the same as $\Z_2$-CW structure on $S^2$ given above. Thus, we obtained that
\begin{align*}
    C^c(\widetilde{S^2(x)})=
    \begin{cases}
        \Z \oplus \Z &\text{  if  } i=0,1\\
        \Z[\Z_2] &\text{  if  } i=2\\
        0 &\text{  otherwise.}
    \end{cases}
\end{align*}
Now, note that 
\[
S^2(x)^{>1} = \{v \in S^2 \mid G_v \neq \{1\}\} = (S^2)^{\mathbb{Z}_2} \cong S^1,
\]
and thus, $ \widetilde{S^2(x)^{>1}} \cong (S^2)^{\mathbb{Z}_2} $. Clearly, CW structure of $\widetilde{S^2(x)^{>1}}$ is equal to the $\Z_2$-CW structure of $\widetilde{S^2(x)}$ at the degree $0$ and $1$, and it is equal to $0$ at the degree $2$. As a result, 
$$
C^c_i(\widetilde{f(x)},\widetilde{f^>(x)})=C^c_i(\widetilde{f(x)})-C_i^c(\widetilde{f^>(x)})=0
$$
for $i=0,1$. Therefore,
\[
C^c(\widetilde{f(x)}, \widetilde{f^{>1}(x)}): C^c(\widetilde{S^2(x)}, \widetilde{S^2(x)^{>1}}) \to C^c(\widetilde{S^2(x)}, \widetilde{S^2(x)^{>1}})
\]  
is given by 
\[
C^c_2(\widetilde{f(x)}): C^c_2(\widetilde{S^2(x)}) \to C^c_2(\widetilde{S^2(x)}),
\]  
and, this maps $ \theta \mapsto -g \cdot \theta $. Thus, we obtain that
\[
C^c(\widetilde{f(x)}, \widetilde{f^{>1}(x)}): \mathbb{Z}[\mathbb{Z}_2] \to \mathbb{Z}[\mathbb{Z}_2] = [-g],
\]  
which is a $ 1 \times 1 $-matrix over $ \mathbb{Z}[\mathbb{Z}_2] $.

Next, we compute the $ \operatorname{Aut}(y) $-CW structure on $ \widetilde{S^2(y)} $, which is the connected component of $ (S^2)^{\mathbb{Z}_2} $ containing the point $ y $. Thus, $ \widetilde{S^2(y)} = \widetilde{S^1} $. An $\operatorname{Aut}(y)$-CW complex structure on $\widetilde{S^1}$ consists of: Two $0$-cells of orbit type $\{1\}$, denoted
$$
  \widetilde{a}, \widetilde{b} \colon \operatorname{Aut}(y) \times e_0 \to \mathbb{R},
$$
Two $1$-cells of orbit type $\{1\}$, denoted
$$
  \widetilde{k}, \widetilde{l} \colon \operatorname{Aut}(y) \times e_1 \to S^1.
$$
Furthermore, $\aut(y) \cong \Z$ because of the following short exact sequence.
\begin{align*}
    1 \to \pi_1(S^1,y) \to \aut(y) \to W\Z_2 \cong \{1\} \to 1
\end{align*}
Thus, the cellular chain complex of $C^c(\widetilde{S^2(y)})$ as a $\aut(y)$-module is as follows.

\begin{align*}
    C^c(\widetilde{S^2(y)})=
    \begin{cases}
        \Z[\Z] \oplus \Z[\Z] &\text{  if  } i=0,1\\
        0 &\text{  otherwise}
    \end{cases}
\end{align*}
It is clear that $C^c(\widetilde{f(y)},\widetilde{f^{>1}(y)})=C^c(\widetilde{f(y)})$ since $(S^1)^{>\Z_2}=\emptyset$. Also, since the induced map $f^{\Z_2}$ is identity map on $S^1$, the cellular map is given in the following way.
\[
C^c(\widetilde{f(y)})=\begin{bmatrix}
  1 & 0\\ 
  0 & 1
\end{bmatrix} - \begin{bmatrix}
  1 & 0\\ 
  0 & 1
\end{bmatrix}=0
\]
Furthermore, by Splitting Theorem \cite[Theorem 4.9]{Weber06}, we know that the abelian group where the universal invariant lies splits as follows.
\begin{align*}
    U^{\Z}_{\Z_2}(f) &= K_0(\phi_{x,w}\text{-end}_{\text{ff}\Z\text{Aut}(x)})\oplus K_0(\phi_{y,v}\text{-end}_{\text{ff}\Z\text{Aut}(y)})\\
    u^{\Z}_{\Z_2}(f)&= C^c(\widetilde{f(x)},\widetilde{f^{>1}(x)}) \oplus C^c(\widetilde{f(y)},\widetilde{f^{>1}(y)})
\end{align*}
Therefore, the universal equivariant Lefschetz invariant is $u^{\Z}_{\Z_2}(f)=[g]$, which is a $1\times 1$-matrix over ${\Z[\Z_2]}$.

Now, we can compute the equivariant functorial Lefschetz invariant $\lambda_{\Z_2}(f)$ which lies in 
\[
\Lambda_{\Z_2}(f) \cong \Z[\pi_1(S^2(x),x)_f] \oplus \Z[\pi_1(S^1(y),y))_f]\cong \Z \oplus \Z[\Z].
\]

The invariant $\lambda_G(f)$ is defined as $\lambda_G(f)_x=\operatorname{tr}_{\Z \aut(x)}(u_G^\Z(X,f)_x)$ in each component. Therefore, $\lambda_{\Z_2}(f)_x = \tr_{\Z\aut(x)}([-g])=0$, and $\lambda_{\Z_2}(f)_y=0$.

As a result, $\lambda_{\Z_2}(f)$ and Klein-Williams invariant $\ell_{\Z_2}(f)$ vanish simultaneously, although the universal invariant does not vanish.
\end{example}

\begin{example} Let $X=S^3$ and $G=\Z_2$, cyclic group of order $2$, with generator $g$. The group acts on $S^3\subseteq \R^4$ in the following way.
\[ 
g \cdot (x_1,x_2,x_3,x_4)=(x_1,x_2,x_3,-x_4) 
\]
Let $f$ be an equivariant map given by
\begin{align*}
    f: S^3 & \to S^3\\
    (x_1,x_2,x_3,x_4) & \mapsto (-x_1,-x_2,-x_3,x_4)
\end{align*}

Similar to previous example, the Klein-Williams invariant $\ell_G(f)$ is equal to 
$$
\overline{R(f)}^{\Z_2} + R(f^{\Z_2}),
$$
and this invariant lies under the following decompositions.
\[
\Omega_0^{\Z_2,\fr}(\Loop_f S^3)= \Z[\pi_1(S^3,x)_f/\Z_2] \oplus \Z[\pi_1((S^3)^{\Z_2},y)_f]\cong \Z \oplus \Z
\]
This is because both $S^3$ and the fixed set 
\[
(S^3)^{\Z_2}=\{((x_1,x_2,x_3,0)\mid x_1^2+x_2^2+x_3^2=1\}=S^2
\]
are simply-connected. Since $\overline{R(f)}^{\Z_2} \in \Z[\pi_1(S^3,x)_f/\Z_2] \cong \Z[\{1\}]$, the reduced an the ordinary Reidemeister traces are the same, and 
$$
\overline{R(f)}^{\Z_2}=R(f)=L(f)[1],
$$
where $L(f)$ is the Lefschetz number of $f$, which is equal to
\[
L(f)=1+(-1)^3\deg(f)=1+(-1)^3(-1)^3=2.
\]
Note that $R(f^{\Z_2})=0$ since induced map $f^{\Z_2}$ on $S^2$ is antipodal map, and hence it has no fixed point. Therefore, $\ell_{\Z_2}(f)=2[1]$.

Now, we will compute the universal equivariant invariant. First, we give the objects and morphisms of the fundamental category $\Pi(\Z_2,S^3)$. Objects of $\Pi(\Z_2,S^3)$ are given as follows.

If $x\in S^3 - (S^3)^{\Z_2}$, then
\begin{align*}
    x: \Z_2/\{1\} & \to S^3\\
    1 & \mapsto x\\
    g & \mapsto g\cdot x
\end{align*}

If $y\in (S^3)^{\Z_2}$, then
\begin{align*}
    y: \Z_2/\Z_2 & \to S^3\\
    1 & \mapsto y
\end{align*}

Morphisms of $\Pi(\Z_2,S^3)$ are listed as follows.
$$
\Mor(x_1,x_2)=\{(\text{id},[\text{path}_{x_1}^{x_2}]),(r,[\text{path}_{x_1}^{x_2}])\},
$$
where $\text{id}:\Z_2/\{1\}\to \Z_2/\{1\}$, and $r:\Z_2/\{1\}\to \Z_2/\{1\}$ such that $r(1)=g$, $r(g)=1$ are $\Z_2$-maps like in the previous example. Also $\text{path}_{x_1}^{x_2}$ denotes any path from $x_1$ to $x_2$ in $S^3$.
$$
\Mor(y_1,y_2)=\{(\text{id},[\text{path}_{y_1}^{y_2}])\},
$$
where $\text{id}:\Z_2/\Z_2\to \Z_2/\Z_2$, and $\text{path}_{y_1}^{y_2}$ is an any path from $y_1$ to $y_2$ in $(S^3)^{\Z_2}\cong S^2$.
$$
\Mor(x,y)=\{(p,[\text{path}_{x}^{y}])\},
$$
where $p:\Z_2/\{1\}\to \Z_2/\Z_2$ such that $p(1)=1$, $p(g)=1$ is $\Z_2$-map.
$$
\Mor(y,x)= \emptyset
$$
because there is no $\Z_2$-map from $\Z_2/\Z_2$ to $\Z_2/\{1\}$.

To obtain $\Z\Pi(\Z_2,S^3)$-chain complex $C^c(\widetilde{S^3})$, first, we will give a $\Z_2$-CW structure of $S^3$. It has two $0$-cells of type $\Z_2$, namely 
$$
a,b: e_0\to (S^3)^{\Z_2},
$$ 
two $1$-cells of type $\Z_2$, namely 
$$
k,l: e_1 \to (S^3)^{\Z_2},
$$
two $2$-cells of type $\Z_2$, namely 
$$
\sigma,\alpha: e_2 \to (S^3)^{\Z_2},
$$
and one $3$-cell of type $\{1\}$, namely the map $\theta$, given as follows.
\begin{align*}
\theta\colon \Z_2 \times e_3 & \to S^3\\
(1,u) & \mapsto u\\
(g,u) & \mapsto g\cdot u
\end{align*}
Note that $S^3(x)=S^3$ and $\aut(x) \cong \Z_2$ by the following short exact sequence.
\begin{align*}
    1 \to \pi_1(S^3(x),x) \to \aut(x) \to W1\cong\Z_2 \to 1
\end{align*}
Therefore, an $\aut(x)$-CW structure on $\widetilde{S^3(x)}$ is the same as $\Z_2$-CW structure on $S^3$ given above. Thus, we obtained that
\begin{align*}
    C^c(\widetilde{S^3(x)})=
    \begin{cases}
        \Z \oplus \Z &\text{  if  } i=0,1,2\\
        \Z[\Z_2] &\text{  if  } i=3\\
        0 &\text{  otherwise.}
    \end{cases}
\end{align*}
Since  
\[
S^3(x)^{>1} = \{v \in S^3 \mid G_v \neq \{1\}\} \cong (S^3)^{\mathbb{Z}_2} = S^2,  
\]  
we have $ \widetilde{S^3(x)^{>1}} \cong S^2 $, whose $ \mathbb{Z}_2 $-CW structure is inherited from the $ \mathbb{Z}_2 $-CW structure of $ \widetilde{S^3(x)} $ in degrees 0, 1, and 2, and it is trivial in degree 3. Therefore, we conclude that the map  
\[
C^c(\widetilde{f(x)}, \widetilde{f^{>1}(x)}): C^c(\widetilde{S^3(x)}, \widetilde{S^3(x)^{>1}}) \to C^c(\widetilde{S^3(x)}, \widetilde{S^3(x)^{>1}})  
\]  
which is given by  
\[
C^c_3(f): C^c_3(\widetilde{S^3(x)}) \to C^c_3(\widetilde{S^3(x)}).  
\]  
Note that $ f $ maps $ \theta $ to $ -\theta $. Thus, $C^c(\widetilde{f(x)},\widetilde{f^{>1}(x)}): \Z[\Z_2] \to \Z[\Z_2] = -[-1]$, which is over $\Z[\Z_2]$. Now, we describe the $ \operatorname{Aut}(y) $-CW structure on $ \widetilde{S^3(y)} \cong \widetilde{S^2} $. The $ \operatorname{Aut}(y) $-CW complex structure of $ \widetilde{S^2} $ consists of: Two $ 0 $-cells of type $ \{1\} $, namely  
\[
a, b: \operatorname{Aut}(y) \times e_0 \to S^2,  
\]  
Two $ 1 $-cells of type $ \{1\} $, namely  
\[
k, l: \operatorname{Aut}(y) \times e_1 \to S^2,  
\]  
Two $ 2 $-cells of type $ \{1\} $, namely  
\[
\sigma, \alpha: \operatorname{Aut}(y) \times e_2 \to S^2.  
\]
Furthermore, we have $ \operatorname{Aut}(y) \cong \{1\} $ due to the following short exact sequence.  
\begin{align*}
    1 \to \pi_1(S^2,y) \to \aut(y) \to W\Z_2 \cong \{1\} \to 1
\end{align*}
Thus, the cellular chain complex of $C^c(\widetilde{S^3(y)})$ is non-equivariant cellular chain which consists of free $\Z$-modules.
\begin{align*}
    C^c(\widetilde{S^3(y)})=
    \begin{cases}
        \Z \oplus \Z &\text{if  } i=0,1,2\\
        0 &\text{otherwise.}
    \end{cases}
\end{align*}
For the same reason as in the previous example, we have  
\[
C^c(\widetilde{f(y)}, \widetilde{f^{>1}(y)}) = C^c(\widetilde{f(y)}). 
\]  
This time, the induced map $ f^{\mathbb{Z}_2} $ is the antipodal map on $ S^2 $. Thus, the cellular map is equal to the following $ \mathbb{Z} $-matrix:  

\[
C^c(\widetilde{f(y)}) =  
\begin{bmatrix}  
0 & 1 \\  
1 & 0  
\end{bmatrix}  
-  
\begin{bmatrix}  
0 & 1 \\  
1 & 0  
\end{bmatrix}  
+  
\begin{bmatrix}  
0 & 1 \\  
1 & 0  
\end{bmatrix}  
= 0.  
\]  
The abelian group $ U_{\mathbb{Z}_2}^\mathbb{Z}(f) $, where the universal invariant $ u_{\mathbb{Z}_2}^\mathbb{Z}(f) $ lies, splits in a similar way to the previous example:  
\[
U^{\mathbb{Z}}_{\mathbb{Z}_2}(f) = K_0(\phi_{x,w}\text{-end}_{\text{ff}\mathbb{Z}\operatorname{Aut}(x)}) \oplus K_0(\phi_{y,v}\text{-end}_{\text{ff}\mathbb{Z}\operatorname{Aut}(y)}),  
\]  
\[
u^{\mathbb{Z}}_{\mathbb{Z}_2}(f) = C^c(\widetilde{f(x)}, \widetilde{f^{>1}(x)}) \oplus C^c(\widetilde{f(y)}, \widetilde{f^{>1}(y)}).  
\]  
Therefore, the universal equivariant Lefschetz invariant is  
\[
u^{\mathbb{Z}}_{\mathbb{Z}_2}(f) = -[-1],  
\]  
which is a $ 1 \times 1 $-matrix over $ \mathbb{Z}[\mathbb{Z}_2] $. As a result, the equivariant functorial Lefschetz invariant is given by  
\[
\lambda_{\mathbb{Z}_2}(f) = -\tr_{\mathbb{Z}\operatorname{Aut}(x)}([-1]) = 1,  
\]  
which lies in  
\[
\Lambda_{\mathbb{Z}_2}(f) \cong \mathbb{Z}[\pi_1(S^3(x),x)_f] \oplus \mathbb{Z}[\pi_1(S^2(y),y)_f] \cong \mathbb{Z} \oplus \mathbb{Z}.  
\]  

One can observe that in the last two example, $\ell_G(f)_x=|WH_x|\lambda_G(f)_x$ for all $\overline{x} \in \text{Is}\Pi(G,X)$. However, this is not always true, and following example this does not hold.
\end{example}

\begin{example} Let $X=S^2$ and $G=\Z_2\times \Z_2$, the Klein group of order $4$, with generators $g$ and $h$. The group acts on $S^2\subseteq \R^3$ in the following way.
\begin{align*}
 g \cdot (x_1,x_2,x_3)=(x_1,x_2,-x_3) \\
h \cdot (x_1,x_2,x_3)=(x_1,x_2,-x_3)   
\end{align*}

Let $f$ be an identity map on $S^2$. We fist calculate the Klein-Williams invariant:

\begin{align*}
    \ell_G(f) &= R(f^G) \oplus \overline{R(f^{\langle h \rangle})}^{W\langle h \rangle} \oplus \overline{R(f^{\langle g \rangle})}^{W\langle g \rangle} \oplus \overline{R(f)}^G\\
    &= \chi(X^G)[1] \oplus \chi(X^{\langle h \rangle})[1] \oplus  \chi(X^{\langle g \rangle})[1] \oplus \chi(X)[1],
\end{align*}
which lies in
\[
    \Z[\pi_1(X^G,x)_f] \oplus \Z[\pi_1(X^{\langle h \rangle},y)_f]/W\langle h \rangle \oplus  \Z[\pi_1(X^{\langle g \rangle},z)_f]/W\langle g \rangle \oplus \Z[\pi_1(X,w)_f]/G.
\]
It is clear that 
\begin{align*}
       \chi(X^G)=\chi(S^0)=2\\
       \chi(X^{\langle h \rangle})=\chi(S^1)=0\\ 
       \chi(X^{\langle g \rangle})=\chi(S^1)=0\\
       \chi(X)=\chi(S^2)=2    
\end{align*}
Therefore, $\ell_G(f)=2[1] \oplus 2[1]$.

Now we will compute the universal invariant $u_G^\Z(X,f)$. Objects of the fundamental category $\Pi(G,X)$ is given as follows.

If $x \in X-X^G$, then
\begin{align*}
    x: G/\{1\} & \to X\\
    k & \mapsto k \cdot x
\end{align*}

If $y \in X^{\langle h\rangle}-X^G$, then
\begin{align*}
    y: G/\langle h \rangle & \to X\\
    1 & \mapsto y\\
    g & \mapsto  g\cdot y
\end{align*}

If $z \in X^{\langle g\rangle}-X^G$, then
\begin{align*}
    z: G/\langle g \rangle & \to X\\
    1 & \mapsto z\\
    h & \mapsto  h\cdot z
\end{align*}

If $w \in X^G$, then
\begin{align*}
    w: G/G & \to X\\
    1 & \mapsto w
\end{align*}

For this example, we will skip the morhpisms of $\Pi(G,X)$, we will only describe the necessary automorphisms later.

Now, we will give a $G$-CW structure of $S^2$. It has two $0$-cells of type $G$, namely 
$$
a,b: e_0\to (S^2)^G,
$$
one $1$-cells of type $\langle h \rangle$, namely 
$$
k: \langle g \rangle \times e_1 \to (S^2)^{\langle h \rangle},
$$
one $1$-cells of type $\langle g \rangle$, namely 
$$
l:\langle h \rangle \times  e_1 \to (S^2)^{\langle g \rangle},
$$
and one $2$-cell of type $\{1\}$, namely 
$$
\theta: G \times e_2 \to S^2.
$$
Clearly, $S^2(x)=S^2$ and $\aut(x) \cong G$ by the following short exact sequence.
\begin{align*}
    1 \to \pi_1(S^2(x),x) \to \aut(x) \to W1\cong G \to 1
\end{align*}
Therefore, an $\aut(x)$-CW structure on $\widetilde{S^2(x)}$ is the same as $G$-CW structure on $S^2$ given above. Thus, we obtained that
\begin{align*}
    C^c(\widetilde{S^2(x)})=
    \begin{cases}
        \Z \oplus \Z & \text{ if  } i=0\\
        \Z\langle g \rangle \oplus \Z\langle h \rangle & \text{  if  } i=1\\
        \Z[G] & \text{  if  } i=2\\
        0 & \text{  otherwise.}
    \end{cases}
\end{align*}
Note that 
$$
S^2(x)^{>1}=\{v\in S^2\mid G_v \neq \{1\}\}\cong (S^2)^{\langle h \rangle} \cup  (S^2)^{\langle g \rangle},
$$
and thus, 
$$
\widetilde{S^2(x)^{>1}} \cong  (S^2)^{\langle h \rangle} \cup  (S^2)^{\langle g \rangle}.
$$
It has the following cellular chain complex.
\begin{align*}
    C^c(\widetilde{S^2(x)^{>1})}=
    \begin{cases}
        \Z \oplus \Z & \text{ if  } i=0\\
        \Z\langle g \rangle \oplus \Z\langle h \rangle & \text{  if  } i=1\\
        0 & \text{  otherwise.}
    \end{cases}
\end{align*}
Therefore, 
$$
C^c(\widetilde{f(x)},\widetilde{f^{>1}(x)}):C^c(\widetilde{S^2(x)},\widetilde{S^2(x)^{>1})} \to C^c(\widetilde{S^2(x)},\widetilde{S^2(x)^{>1}})
$$
is given by 
$$
C^c_2(f): C^c_2(\widetilde{S^2(x)}) \to C^c_2(\widetilde{S^2(x)}),
$$
which maps $\theta \mapsto  \theta$ since $f=\text{id}$. Thus, 
$$
C^c(\widetilde{f(x)},\widetilde{f^{>1}(x)}): \Z[G] \to \Z[G] = [1],
$$ 
over $\Z[G]$. Now, we will consider $\widetilde{S^2(y)}$ and its $\aut(y)$-CW structure. A $W\langle h \rangle$-CW structure on 
$$
S^2(y)=(S^2)^{\langle h \rangle} \cong S^1,
$$
given by; two $0$-cells of type $W \langle h \rangle \cong \langle g \rangle$, namely 
$$
a,b: e_0\to (S^2)^G,
$$
and one $1$-cell of type $\{1\}$, namely 
$$
k: \langle g \rangle \times e_1 \to (S^2)^{\langle h \rangle}.
$$
Then, the $\aut(y)$-CW structure of $\widetilde{S^2(y)}\cong \R$ is the following. It has two $0$-cells of type $\langle g \rangle$,
$$
\widetilde{a},\widetilde{b}: \Z \times e_0\to (S^2)^G,
$$
and one $1$-cell of type $\{1\}$, namely 
$$
\widetilde{k}: \aut(y) \times e_1 \to (S^2)^{\langle h \rangle}.
$$
Note that 
$$
\aut(y)=\{(\text{id},[\text{p}_{y}^{a}*s^n*\text{p}_{a}^{y}])\mid n\in \Z\} \cup \{(\text{r},[\text{p}_{y}^{a}*s^n*\text{p}_{a}^{y}])\mid n\in \Z\},
$$
where $\text{id},r\colon G/\langle h \rangle\to G/\langle h  \rangle$ such that $r(1)=g$ are $G$-maps. Also, $s$ is the generator of $\pi_1((S^2)^{\Z_2},a)=\pi_1(S^1,a)$, which is the loop at $a$, going once around $S^1$ counterclockwise, and $\text{p}_{a}^{y'}$ denotes path from $a$ to $y'$ in $S^1$. Furthermore, there exists short exact sequence of the groups.
\[
1 \to \pi_1(S^2(y),y)\cong \pi_1(S^1,y) \to \aut(y) \to W \langle h \rangle \cong \Z_2 \to 1
\]
As a result, the cellular chain complex of $S^2(y)$ is given by
\begin{align*}
    C^c(\widetilde{S^2(y)})=
    \begin{cases}
        \Z[\Z] \oplus \Z[\Z] &\text{  if  } i=0\\
        \Z \aut(y) &\text{  if  } i=1\\
        0 &\text{  otherwise.}
    \end{cases}
\end{align*}
Now, we aim to compute the cellular map $C^c(\widetilde{f(y)},\widetilde{f^{>\langle h \rangle}(y)})$. First, observe that 
$$
S^2(y)^{>\langle h \rangle}=\{ v \in (S^2)^{\langle h \rangle}\mid G_v \neq \langle h \rangle \}=\{a,b\} =S^0.
$$
This implies that $\widetilde{S^2(y)^{>\langle h \rangle}}=\{\widetilde{a},\widetilde{b}\}$. Therefore,
\begin{align*}
    C^c(\widetilde{S^2(y)^{> \langle h \rangle}})=
    \begin{cases}
        \Z[\Z] \oplus \Z[\Z] &\text{  if  } i=0\\
        0 &\text{  otherwise.}
    \end{cases}
\end{align*}
As a result, $C^c(\widetilde{f(y)},\widetilde{f^{>\langle h \rangle}(y)})$ is given by 
$$
C^c_1(f): C^c_1(\widetilde{S^2(y)}) \to C^c_1(\widetilde{S^2(y)}),
$$
which maps $\widetilde{k} \mapsto  \widetilde{k}$ since $f=\text{id}$, and thus;
$$
C^c(\widetilde{f(y)},\widetilde{f^{>\langle h \rangle}(y)})=(-1)[1]=-[1],
$$
where $[1]$ is a $1 \times 1$-matrix over $\Z\aut(y)$.

The case for the $\aut(z)$-CW structure on $\widetilde{S^2(z)}$ is similar to the previous case for $\widetilde{S^2(y)}$. In particular, their cellular chain complexes are the same because note that 
$$
S^2(z)=(S^2)^{\langle g \rangle}\cong S^1,
$$
and it has the $W\langle g \rangle \cong \Z_2$-action, which is the same $W\langle h \rangle$-action on $S^2(y)$. Also, $\aut(z) \cong \aut(y)$. As a result,
\begin{align*}
    C^c(\widetilde{S^2(z)})=
    \begin{cases}
        \Z[\Z] \oplus \Z[\Z] &\text{  if  } i=0\\
        \Z \aut(z) &\text{  if  } i=1\\
        0 &\text{  otherwise.}
    \end{cases}
\end{align*}
Moreover, 
$$
S^2(z)^{>\langle g \rangle}=\{ v \in (S^2)^{\langle g \rangle}\mid G_v \neq \langle g \rangle \}=\{a,b\} =S^0.
$$
Thus, similar to $C^c(\widetilde{S^2(y)^{> \langle h \rangle}})$, we have the cellular chain complex
\begin{align*}
    C^c(\widetilde{S^2(z)^{> \langle g \rangle}})=
    \begin{cases}
        \Z[\Z] \oplus \Z[\Z] &\text{  if  } i=0\\
        0 &\text{  otherwise.}
    \end{cases}
\end{align*}
Therefore, $C^c(\widetilde{f(z)},\widetilde{f^{>\langle g \rangle}(z)})$ is given by 
$$
C^c_1(f): C^c_1(\widetilde{S^2(z)}) \to C^c_1(\widetilde{S^2(z)}),
$$
which maps $\widetilde{l} \mapsto  \widetilde{l}$, where 
$$
\widetilde{l}: \aut(z) \times e_1 \to (S^2)^{\langle g \rangle}
$$
is the $1$-cell of $\aut(z)$-CW structure of $\widetilde{S^2(z)}$. Thus,
$$
C^c(\widetilde{f(z)},\widetilde{f^{>\langle g \rangle}(z)})=(-1)^1[1]=-[1],
$$
where $[1]$ is a $1 \times 1$-matrix over $\Z\aut(z)$.

The last case is the $\aut(w)$-CW structure on $\widetilde{S^2(w)}$. In this case, $(S^2)^G\cong S^0$, which is not connected, and we have two connected components; $S^2(a)$ denotes $\{a\}$, the component which contains the point (0-cell) $a$, and $S^2(b)$ denotes $\{b\}$. Thus, $\widetilde{S^2(w)}$ is the single point. Note that $\aut(w)= \{1\}$ by the following short exact sequence.
\[
1 \to \pi_1(S^2(w),w) \to \aut(w) \to WG=\{1\}\to 1
\]
Therefore, $C^c(\widetilde{S^2(w)})$ is a free $\Z$-chain complex, described as follows.
\begin{align*}
C^c(\widetilde{S^2(a)})=C^c(\widetilde{S^2(b)})=
    \begin{cases}
        \Z &\text{if  } i=0\\
        0 &\text{otherwise.}
    \end{cases}
\end{align*}
It is clear that $C^c(\widetilde{f(w)},\widetilde{f^{> G}(w)})=C^c(\widetilde{f(w)})=[1]$, where $[1]$ is a $1 \times 1$-matrix over $\Z$ for $w=a,b$.

To sum up, the universal equivariant Lefschetz invariant for this example is the following.

\begin{align*}
    U^{\Z}_{G}(f) &= K_0(\phi_{x,v_x}\text{-end}_{\text{ff}\Z\text{Aut}(x)})\oplus K_0(\phi_{y,v_y}\text{-end}_{\text{ff}\Z\text{Aut}(y)}) \oplus\\
    &K_0(\phi_{z,v_z}\text{-end}_{\text{ff}\Z\text{Aut}(z)})\oplus K_0(\phi_{w,v_w}\text{-end}_{\text{ff}\Z\text{Aut}(w)})\\
    u^{\Z}_{G}(f)&= [1]_{\Z\aut(x)} \oplus -[1]_{\Z\aut(y)} \oplus -[1]_{\Z \aut(z)} \oplus ([1]_\Z + [1]_\Z)
\end{align*}
The functorial equivariant Lefschetz invariant 
\[
\lambda_G(f) \in \Lambda_G(f)= \bigoplus_{\substack{\overline{x} \in \text{Is}\Pi(G,X)\\X^H(x)=X^H(f(x))}} \Z\pi_1(X^H(x),x)_f
\] 
is given by
\begin{align*}
    &\lambda_G(f)_x=\tr_{\Z\aut(x)}([1])=1\\
    &\lambda_G(f)_y=\tr_{\Z\aut(y)}(-[1])=-1\\
    &\lambda_G(f)_z=\tr_{\Z\aut(z)}(-[1])=-1\\
    &\lambda_G(f)_w=\tr_{\Z\aut(w)}([1]+[1])=2.
\end{align*}
Note that these correspond to  
\[
\lambda_G(\text{id})_p = \chi(WH_p \backslash X^H(p), WH_p \backslash X^{>H}(p)),
\]
for all $ p \in \text{Is} \Pi(G, X) $. As a result, we obtain that
\begin{align*}
    &\lambda_G(f)_x=\chi(W1 \backslash X^1,WH_p \backslash X^{>1})=\chi(G \backslash S^2)- \chi(G\backslash (S^2)^{>1})=1-0=1 \\
    &\lambda_G(f)_y=\chi(W\langle h \rangle \backslash S^1,W\langle h \rangle \backslash S^0)=1-2=-1\\
    &\lambda_G(f)_z=\chi(W\langle g \rangle \backslash S^1,W\langle g \rangle \backslash S^0)=1-2=-1\\
    &\lambda_G(f)_w=\chi(WG \backslash (S^2)^G,WG \backslash (S^2)^{>1})=\chi(S^0, \emptyset)=2.
\end{align*}
\end{example}

\bibliographystyle{plain}

\end{document}